\documentclass[11pt]{amsart}

\usepackage{amssymb,latexsym,amsmath}
\usepackage{bbm}
\usepackage[all]{xy}
\usepackage{graphicx}

\usepackage{enumerate}
\usepackage{amscd}
\usepackage{graphicx, array}
\usepackage{longtable}
\usepackage{epstopdf}
\usepackage{xcolor}
\usepackage{tikz-cd}
\usepackage{cite}
\usetikzlibrary{arrows.meta}
\usepackage{wrapfig}
\usetikzlibrary{calc}
\usetikzlibrary{backgrounds}
\usetikzlibrary{automata}
\usetikzlibrary{positioning}
\usepackage{scalerel}

\usepackage{amsmath,amssymb,amsfonts,amscd,euscript}
\usepackage{xy} \xyoption{all}

\usepackage{tikz}
\usepackage{color,graphics}

\DeclareMathAlphabet{\Ma}{U}{msa}{m}{n}
\DeclareMathAlphabet{\Mb}{U}{msb}{m}{n}
\DeclareMathAlphabet{\Meuf}{U}{euf}{m}{n}
\DeclareSymbolFont{ASMa}{U}{msa}{m}{n}
\DeclareSymbolFont{ASMb}{U}{msb}{m}{n}

\def\mr #1.{\mathrm{#1\,}}
\def\mrt #1.{\mathrm{\mbox{\tiny #1\,}}}
\def\mt #1.{{\mbox{\tiny $#1$}}}
\def\ms #1.{{\mbox{\small $#1$}}}

\def\C{\mathbb{C}}    
\def\N{\mathbb{N}}
\def\ol{\overline}
\def\ul{\underline}

\newcommand{\Lab}{L^{{\mathrm{ab}}}}

\newcommand{\alrtimes}{\rtimes^{\mathrm{alg}}}

\def\1{\mathbbm 1}

\def\1{\mathbbm 1}
\def\C{\Mb{C}}
\def\ot #1.{{\got{#1}}}
\def\got#1{\Meuf{#1}}
\def\al #1.{{\mathcal{#1}}}

\theoremstyle{plain}            
\newtheorem{theorem}{Theorem}[section]
\newtheorem*{maintheorem*}{Main Theorem}
\newtheorem*{maintheorem.}{Main Theorem~1'}

\newtheorem{proposition}[theorem]{Proposition}
\newtheorem{lemma}[theorem]{Lemma}
\newtheorem{corollary}[theorem]{Corollary}
\newtheorem{openproblem}{Open Problem}

\theoremstyle{definition}       
\newtheorem{definition}[theorem]{Definition}

\theoremstyle{definition}       
\newtheorem{definitions}[theorem]{Definitions}

\theoremstyle{remark}
\newtheorem{remark}[theorem]{Remark}
\newtheorem{example}[theorem]{Example}

\newcommand{\V}{\mathrm{V}}

\newcommand{\Aut}{\mathrm{Aut}}

\newcommand{\Sour}{\mathrm{Sour}}
\newcommand{\Sink}{\mathrm{Sink}}

\newcommand{\Typ}{\mathrm{Typ}}

\newcommand{\typ}{\mathrm{typ}}

\newcommand{\pHomeo}{\mathrm{\bf pHomeo}}

\def\mr #1.{\mathrm{#1\,}}
\def\mrt #1.{\mathrm{\mbox{\tiny #1\,}}}
\def\mt #1.{{\mbox{\tiny $#1$}}}
\def\ms #1.{{\mbox{\small $#1$}}}

\def\C{\mathbb{C}}    
\def\Z{\mathbb{Z}}
\def\N{\mathbb{N}}

\newfont{\Kcal}{cmsy6 scaled 1000}
\newfont{\Kgot}{eufm6 scaled 1000}

\def\Kbegin{\begin{equation} \left. \begin{array}{rcl}}
\def\Kend{\end{array} \right\} \end{equation}}

\DeclareMathSymbol{\hsemi}{\mathord}{ASMb}{"6E}
\newcommand{\semi}[2]{\mbox{$#1\kern.1em\hsemi\kern.1em#2$}}

\def\vplatz#1{{\rule{0mm}{#1}}}

\def\LA{\left\langle\bgroup}
\def\LE{\left[\bgroup}
\def\LG{\left\{\bgroup}
\def\LR{\left(\bgroup}
\def\RA{\egroup^{\rule{0mm}{0mm}}\right\rangle}
\def\RE{\egroup^{\rule{0mm}{2mm}}\right]}
\def\RG{\egroup^{\rule{0mm}{2mm}}\right\}}
\def\RR{\egroup^{\rule{0mm}{2mm}}\right)}
\def\Ldummy{\left.\bgroup}
\def\Rdummy{\egroup^{\rule{0mm}{2mm}}\right.}

\def\ccr#1{\mbox{{\rm CCR$\left({#1}^{\vplatz{1.5mm}}\right)$}}}

\def\ccr #1,#2.{\overline{\Delta(#1,\,#2)}}

\def\b #1.{{\bf #1}}
\def\cross#1.{\mathrel{\mathop{\times}\limits_{#1}}}
\def\C{\Mb{C}}
\def\N{\Mb{N}}

\def\wwh #1.{\widehat{#1}}
\def\wt #1.{\widetilde{#1}}

\def\cross #1.{\mathrel{\raise 3pt\hbox{$\mathop\times\limits_{#1}$}}}
\def\set #1,#2.{\left\{\,#1\;\bigm|\;#2\,\right\}}
\def\b #1.{{\bf #1}}
\def\ker{{\rm Ker}\,}
\def\coker{{\rm Coker}\,}

\def\rn#1.{\romannumeral{#1}}

\def\s #1.{_{\smash{\lower2pt\hbox{\mathsurround=0pt $\scriptstyle #1$}}\mathsurround=3pt}}
\def\bra #1,#2.{{\left\langle #1,\,#2\right\rangle_{\al A.}}}

\def\XP#1!{\renewcommand{\baselinestretch}{.7}\marginpar{{\footnotesize #1}\hfil}
\renewcommand{\baselinestretch}{1.5}}
\def\XB{\marginpar{
{\footnotesize\bf Change~starts----}\lower 11pt\hbox{\mathsurround=0pt$
\!\!\displaystyle{
\Bigg\downarrow}$\mathsurround=3pt}}}
\def\XE{\marginpar{{\footnotesize\bf Change~ends-----}\raise 10pt\hbox{\mathsurround=0pt$
\!\!\displaystyle{
\Bigg\downarrow}$\mathsurround=3pt}}}

\DeclareMathSymbol{\hsemi}{\mathord}{ASMb}{"6E}

\def\LA{\left\langle\bgroup}
\def\LE{\left[\bgroup}
\def\LG{\left\{\bgroup}
\def\LR{\left(\bgroup}
\def\RA{\egroup^{\rule{0mm}{0mm}}\right\rangle}
\def\RE{\egroup^{\rule{0mm}{2mm}}\right]}
\def\RG{\egroup^{\rule{0mm}{2mm}}\right\}}
\def\RR{\egroup^{\rule{0mm}{2mm}}\right)}
\def\Ldummy{\left.\bgroup}
\def\Rdummy{\egroup^{\rule{0mm}{2mm}}\right.}

\def\ccr#1{\mbox{{\rm CCR$\left({#1}^{\vplatz{1.5mm}}\right)$}}}

\newcommand{\F}{\mathbb F}

\newcommand{\act}{\curvearrowright}

\def\ccr #1,#2.{\overline{\Delta(#1,\,#2)}}

\def\b #1.{{\bf #1}}
\def\cross#1.{\mathrel{\mathop{\times}\limits_{#1}}}
\def\C{\Mb{C}}
\def\N{\Mb{N}}

\def\wwh #1.{\widehat{#1}}
\def\wt #1.{\widetilde{#1}}

\def\cross #1.{\mathrel{\raise 3pt\hbox{$\mathop\times\limits_{#1}$}}}
\def\set #1,#2.{\left\{\,#1\;\bigm|\;#2\,\right\}}
\def\b #1.{{\bf #1}}
\def\ker{{\rm Ker}\,}

\def\rn#1.{\romannumeral{#1}}

\def\s #1.{_{\smash{\lower2pt\hbox{\mathsurround=0pt $\scriptstyle #1$}}\mathsurround=3pt}}
\def\bra #1,#2.{{\left\langle #1,\,#2\right\rangle_{\al A.}}}

\def\XP#1!{\renewcommand{\baselinestretch}{.7}\marginpar{{\footnotesize #1}\hfil}
\renewcommand{\baselinestretch}{1.5}}
\def\XB{\marginpar{
{\footnotesize\bf Change~starts----}\lower
11pt\hbox{\mathsurround=0pt$ \!\!\displaystyle{
\Bigg\downarrow}$\mathsurround=3pt}}}
\def\XE{\marginpar{{\footnotesize\bf Change~ends-----}\raise 10pt\hbox{\mathsurround=0pt$
\!\!\displaystyle{ \Bigg\downarrow}$\mathsurround=3pt}}}

\title[Graph $C^*$-algebras]{An introduction to separated graphs and their type semigroups}
\author{Pere Ara}
\address{Department of Mathematics, Universitat Aut\`onoma de Barcelona, 08193 Bellaterra (Barcelona), Spain}
\email{pere.ara@uab.cat}
\date{\today}

\subjclass[2020]{46L80, 16S88, 20M18}

\keywords{Graph $C^*$-algebra, Self-similar action, Separated graph, Leavitt path algebra, type semigroup}

\thanks{The author was partially supported by the Spanish State Research Agency (grant No.\ PID2023-147110NB-I00 and CEX2020-001084-M), and by the Comissionat per Universitats i Recerca de la Generalitat de Catalunya (grant No.\ 2021-SGR-01015). }

\begin{document}
\maketitle

\begin{abstract}
We introduce $C^*$-algebras associated with directed graphs, along with two generalizations of this concept, namely Exel-Pardo $C^*$-algebras associated with a self-similar action of a group on a directed graph, and the $C^*$-algebras associated with separated graphs. These constructions have in common that they have a dynamical behavior, being the groupoid $C^*$-algebras associated to certain topological groupoids, which are built from the combinatorial structure. An important invariant one may associate to these dynamical systems is the so-called type semigroup. We will find a formula to compute the type semigroup for a general self-similar action of a group on a row-finite graph $E$ without sources, following a recent paper by Kwaśniewski, Meyer and Prasad, and for any finite bipartite separated graph, following a paper by Exel and the author. In addition, we will review various results concerning the structure of the type semigroup for different dynamical systems.  
\end{abstract}
\tableofcontents

\section{Introduction}

\vskip1cm

These notes are an expanded version of the lecture notes for the course ``Introduction to separated graphs" that I delivered at the Summer School Luis Santaló ``Structure and approximation of $C^*$-algebras", held at the UIMP, Santander, in July 2025.  

Graph $C^*$-algebras are a central object in modern operator algebra theory, connecting functional analysis, topology, dynamics, and even number theory. Their importance comes from how they translate combinatorial data (graphs) into rich analytic structures ($C^*$-algebras), creating a powerful bridge between discrete and continuous mathematics. Graph $C^*$-algebras associated to directed graphs generalize the important class of Cuntz-Krieger algebras \cite{CK}, and provide important models for the classification of $C^*$-algebras. In turn, they have been generalized in several directions, some of which will be considered here.
 
In these notes, we introduce $C^*$-algebras associated with directed graphs, along with two generalizations of this concept, namely Exel-Pardo $C^*$-algebras associated with a self-similar action of a group on a directed graph, and the $C^*$-algebras associated with separated graphs. These constructions have in common that they have a dynamical behavior, being the groupoid $C^*$-algebras associated to certain topological groupoids, which are built from the combinatorial structure. As such, an important invariant one may associate to the dynamical systems is the so-called {\it type semigroup}. This has its origins in the work of Tarski \cite{tarski38, TomWagon} on paradoxical decompositions, where the type semigroup serves to express paradoxical decompositions (or their absence) in an algebraic language.
Hence the type semigroup is a purely dynamical device that one may try to directly compute from the combinatorial object. We show formulas for the computation of the type semigroup for a general 
self-similar action of a group on a row-finite graph $E$ without sources, following the recent paper \cite{KMP25}, and for any finite bipartite separated graph, following \cite{ae14}. 

There is a natural map from the type semigroup to the Murray-von Neumann monoid of projections of the associated $C^*$-algebras, which in some cases, turns out to be a monoid isomorphism. 
We advocate here that this map is a non-stable version of the well-known map $H_0(\mathcal G) \to K_0 (C^*(\mathcal G))$ that shows up in some formulations of Matui's Conjectures \cite{matui, BDGW}, see Remark \ref{rem:HKConjecture}.

We will survey two important results of Wehrung \cite{weh} concerning the type semigroup. The first asserts that every countable conical refinement monoid can be represented as the type monoid of an 
action of a countable group on a locally compact Hausdorff zero-dimensional metrizable space (Theorem \ref{thm:wehrung-realization}). The second asserts that the type semigroup associated to the action of any supramenble group on a locally compact Hausdorff zero-dimensional metrizable space is always strongly separative (Theorem \ref{thm:supramenable}), which is a powerful cancellation condition.

In these notes, I intend to follow a lively, direct expository style, trying to avoid technical conditions that may obscure the essence of the arguments. 
I have included several concrete examples through the notes in order to illustrate the concepts and results. 

\vskip2cm


\section{Graph $C^*$-algebras}
\label{sect:graphs}

\vskip1cm


\subsection{Graphs}

We start by fixing notation for graphs, following Raeburn's book \cite{rae}. We warn the reader that different notations are used in other sources.

\begin{definitions}
	\label{defs:graphs} A {\it directed graph} $E=(E^0,E^1,r,s)$ consists of two countable sets $E^0$ (the vertices) and $E^1$ (the edges) and functions $r,s\colon E^1\to E^0$ giving the range and the source of an edge $e\in E^1$.   

A {\it path} in $E$ is given by a finite string of edges $\gamma = e_1\cdots e_n $ such that $s(e_i)= r(e_{i+1})$ for all $1\le i <n$. The integer $n$ is called the {\it length} of the path, denoted by
$|\gamma |$. Vertices are considered as paths of length $0$, and are called {\it trivial paths}.  
We set $s(\gamma) = s(e_n)$ and $r(\gamma ) = r(e_1)$. We also will consider {\it infinite paths} $\gamma = e_1 e_2\cdots $. In this case, we set $r(\gamma) = r(e_1)$.
\end{definitions}

We will deal here with {\it row-finite graphs}, which are the graphs $E$ such that the adjacency matrix $A_E$,
$$A_E(v,w) = \# \{ e\in E^1 : r(e) = v, s(e) = w \}$$ 
has a finite number of nonzero entries at each row. This means that the vertices in $E$ receive only finitely many arrows from the graph.
Of course, this includes all {\it finite graphs}, which by definition are those graphs $E$ such that $E^0$ and $E^1$ are both finite sets. 

A {\it source} in a directed graph $E$ is a vertex $v\in E^0$ such that $v$ does not receive any edge, that is, such that $r^{-1}(v) = \emptyset$. Similarly, a {\it sink} is a vertex $v\in E^0$ such that $s^{-1}(v)= \emptyset$. We denote the set of all sources of $E$ by $\Sour (E)$, and the set of all sinks by $\Sink (E)$.

\subsection{Cuntz-Krieger $E$-families}

Recall that a {\it projection} in a $*$-ring (or in a $*$-semigroup) is a self-adjoint idempotent element $p$, that is, an element $p$ such that $p=p^*= p^2$. A {\it partial isometry} is an element $w$ such that
$w= ww^*w$. Note that $w$ is a partial isometry if and only if $w^*$ is a partial isometry, and that, for a partial isometry $w$, the elements $ww^*$ and $w^*w$ are both projections. 

\medskip
  
\begin{definition}
Let $E$ be a row-finite graph. A {\it Cuntz-Krieger $E$-family} $\{S,P\}$ in a $C^*$-algebra $A$ is a set $\{ P_v : v\in E^0\}$ of mutually orthogonal projections and a set $\{S_e: e\in E^1\}$ of partial isometries such that 
  \begin{enumerate}
  	\item[(CK1)] $S_e^*S_e = P_{s(e)}$ for all $e\in E^1$,
  	\item[(CK2)] $P_v = \sum_{e\in r^{-1} (v)} S_eS_e^*$ whenever $v$ is not a source.  
  \end{enumerate}
\end{definition}
  
  \begin{definition}
  	\label{def:graphCalg}
 Let $E$ be a row-finite graph. Then the {\it graph $C^*$-algebra} $C^*(E)$ of $E$ is the $C^*$-algebra generated by a universal Cuntz-Krieger $E$-family $\{s,p\}$.    	
  	\end{definition}
  
  Specifically, the universal property described in the above definition says that for any Cuntz-Krieger $E$-family $\{S,P\}$ in any $C^*$-algebra $A$, there exists a unique $*$-homomorphism $\phi\colon C^*(E)\to A$ such that $\phi (p_v)= P_v$ for all $v \in E^0$ and $\phi (s_e)= S_e$ for all $e\in E^1$.  
  
  There is another way to define a graph $C^*$-algebra, using a $*$-algebra called the {\it Leavitt path algebra}
  of $E$. Although this algebra can be indeed defined over an arbitrary commutative unital ring, we will use here only algebras over the complex numbers. We consider complex $*$-algebras, which are algebras over $\C$ with an involution $*$ such that $(\lambda a)^* = \ol{\lambda}a^*$ for all $\lambda\in \C$ and $a\in A$.  
   
    \begin{definition}
   	\label{def:graphCalg}
   	Let $E$ be a row-finite graph. Then the {\it Leavitt path algebra} $L(E)$ is the complex $*$-algebra with generators $\{ v, e :  v\in E^0, e\in E^1 \}$, subject to the
   	following relations:
   	\begin{enumerate}
   		\item[] (V)\ \ $vw = \delta_{v,w}v$ \ and \ $v=v^*$ \ for all $v,w \in E^0$ ,
   		\item[] (E)\ \ $r(e)e=es(e)=e$ \ for all $e\in E^1$ ,
   		\item[] (CK1)\ \ $e^*f=\delta _{e,f} s(e)$ \ for all $e,f\in E^1$, and
   		\item[] (CK2)\ \ $v=\sum _{ e\in r^{-1}(v) }ee^*$ \ for every $v\in E^0\setminus \Sour (E)$.
   	\end{enumerate}
   
   The graph $C^*$-algebra is the {\it enveloping $C^*$-algebra} of $L(E)$, that is, it is a $C^*$-algebra with a $*$-homomorphism $L(E)\to C^*(E)$ such that every $*$-homomorphism $\tau\colon L(E)\to A$ to a $C^*$-algebra $A$ factors through $C^*(E)$.  
   	\end{definition}

It turns out that $L(E)$ is naturally a $*$-subalgebra of $C^*(E)$ and $C^*(E)$ is a completion of $L(E)$ with respect to a suitable norm. See \cite{AAS} for a comprehensive introduction to the theory of Leavitt path algebras.

\subsection{Examples}

We offer here some examples of finite graphs, along with their corresponding graph $C^*$-algebras. A primary example in the theory consists of the $n$-rose graph $R_n$, as follows:
   		$$R_n \ \ =  \ \ \ \ \ \   \ \ 
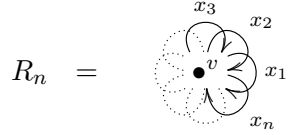
\xymatrix{ \bullet^v \ar@(ur,dr)^{x_1} \ar@(u,r)^{x_2} \ar@(ul,ur)^{x_3} \ar@{.}@(l,u) \ar@{.}@(d,l) \ar@{.}@(dl,ul) \ar@{.}@(dr,dl) \ar@(r,d) ^{x_n}& }$$
$R_n$ has a single vertex and $n$ edges $x_1,x_2,\dots ,x_n$. For $n\ge 2$, the graph $C^*$-algebra $C^*(R_n)$ is the famous {\it Cuntz algebra} $\mathcal O _n$ \cite{Cuntz77}. Explicitly, $\mathcal O _n$ is the $C^*$-algebra generated by $n$ isometries $x_1,\dots ,x_n$ with orthogonal ranges and whose final projections add up to $1$:
$$x_i^*x_j = \delta_{ij} 1,\qquad \sum_{i=1}^n x_ix_i^*= 1.$$
These algebras were introduced by Cuntz \cite{Cuntz77, Cuntz81}, and have played and still continue to play a very important role in the theory of $C^*$-algebras. Similar algebras, known nowadays as {\it Leavitt algebras}, were built by Leavitt \cite{Lea57, Lea62} in order to study the (failure of the) IBN property for rings. Indeed, for each $1\le m  < n$ and for any field $K$, Leavitt constructed an algebra $L_K(m,n)$ with the property that $L_K(m,n)^m \cong L_K(m,n)^n$ as right $L_K(m,n)$-modules, but such isomorphism does not hold for smaller integers. Moreover the algebras $L_K(m,n)$ are universal satisfying the stated isomorphism. It is worth to mention that the work of Cuntz was done independently of the work of Leavitt. Observe that Cuntz algebras only deal with the case $m=1$, that is, the Leavitt path algebra $L(R_n)$ coincides with the algebra $L_{\C} (1,n)$.  Some years later, McClanahan \cite{MCl92, MCl93} constructed $C^*$-algebras $U_{m,n}^{\text{nc}}$ which are related to Leavitt algebras $L_\C (m,n)$. Namely $U_{m,n}^{\text{nc}}$ is the universal $C^*$-algebra generated by the entries of a unitary $m\times n $ matrix $U=(u_{ij})$.
The work of McClanahan (see also \cite{Brown}) was also done independently of the work by Leavitt. 

In the special case $n=1$, we have the single edge graph
$$R_1 \ \ = \ \ 
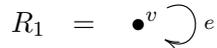
\xymatrix{{\bullet}^{v} \ar@(ur,dr) ^e}$$
whose graph $C^*$-algebra is $C(\mathbb T)$, the algebra of continuous functions on the circle $\mathbb T$. 
The $C^*$-algebra $C(\mathbb T)$ behaves in a different way than the Cuntz algebras $\mathcal O _n$ for $n\ge 2$. The former is commutative and has lots of ideals, the latter is highly non-commutative and is a {\it simple} $C^*$-algebra, that is, it does not contain non-trivial closed two-sided ideals. 

Other interesting and standard examples of graph $C^*$-algebras include the $C^*$-algebra $C^*(A_n)$ associated to the $n$-line quiver  
  $$ \index{$A_n$} A_n \ \ = \ \ \xymatrix{{\bullet}^{v_1} \ar [r] ^{e_1} & {\bullet}^{v_2} \ar [r] ^{e_2} & {\bullet}^{v_3} \ar@{.}[r] &
	{\bullet}^{v_{n-1}} \ar [r] ^{e_{n-1}} & {\bullet}^{v_n}} $$
It is not difficult to show that $C^*(A_n)\cong M_n(\C)$, the $C^*$-algebra of all $n\times n$ matrices with coefficients in $\C$. 

Finally the graph
$$ E_T \ \ = \ \ \xymatrix{ {}^e \hskip-.3in & {\bullet}^u \ar@(ul,dl) &{\bullet}^v \ar[l]_f }   \ \ .$$
gives rise to the {\it Toeplitz $C^*$-algebra}, that is, the $C^*$-algebra of operators on $\ell^2(\N)$ generated by the unilateral shift $T(f_n)= f_{n+1}$, where $\{f_n : n\in \N \}$ is the standard  orthonormal basis of $\ell^2 (\N)$. 
   	
Another interesting source of examples comes from the relations with non-commutative geometry, in particular quantum spheres can be modeled using graph $C^*$-algebras, see \cite{Dandrea} for a nice survey.    
\subsection{When is a graph $C^*$-algebra simple?}

We determine here when a graph $C^*$-algebra $C^*(E)$ is simple. We will deduce in particular that $\mathcal O _n$ is simple for all $n\ge 2$. 

For a path $\mu = e_1e_{2} \cdots e_n$ in $E$, set $s_\mu= s_{e_1}s_{e_{2}}\cdots s_{e_n}$. Thanks to the relation (CK1), $C^*(E)$ is the closed linear span of the family of elements $s_\mu s_\nu^*$, where $\mu,\nu$ range on all pairs of paths such that $s(\mu) = s(\nu)$. 

A {\it cycle} in $E$ is a non-trivial closed path $\mu = e_1\cdots e_n$ (i.e. $r(\gamma ) = s(\gamma)$) such that $s(e_i)\ne s(e_j)$ for all $i\ne j$. We say that a cycle $\mu$ as above {\it has an entry} if there is an edge $f$ and $1\le i\le n$ such that $r(f)= r(e_i)$ and $f\ne e_i$.   

We first state two important uniqueness results for graph $C^*$-algebras. For the proof see e.g. \cite{rae}.

\begin{theorem} (The Cuntz-Krieger uniqueness theorem)
	\label{thm:CKuniqueness}  Suppose $E$ is a row-finite directed graph in which every cycle has an entry. 
	If $A$ is a $C^*$-algebra and $\varphi \colon C^*(E) \to A$ is a $*$-homomorphism such that $\varphi (p_v)\ne 0$ for all $v\in E^0$, then $\varphi$ is injective. 
	\end{theorem}

For the second theorem, we need the notion of the gauge action of $\mathbb T$ on $C^*(E)$. For $z\in \mathbb T$, this action is defined on generators by $\gamma_z(p_v)= p_v$ for all $v\in E^0$, $\gamma_z (s_e) = z s_e$ and
$ \gamma_z (s_e^*) = \ol{z} s_e^*$ for all $e\in E^1$. 

\begin{theorem} (The gauge-invariant uniqueness theorem)
	\label{thm:gauge-invariant}
	Let $E$ be a row-finite directed graph, and suppose that $B$ is a $C^*$-algebra with a continuous action of $\mathbb T$. Then every equivariant $*$-homomorphism $\varphi \colon C^*(E)\to B$ such that $\varphi (p_v)\ne 0$ for all $v\in E^0$ is injective. 
\end{theorem}

We can now discuss the simplicity of graph $C^*$-algebras. It turns out that certain subsets of vertices of a graph $E$ give rise to ideals of $C^*(E)$. Hence to get a simple $C^*$-algebra we need to get rid of these special subsets. Let $E^{\infty}$ be the set of infinite paths $\mu = e_1e_2e_3\cdots $ on $E$, and let $E^{\le \infty}$ be the set of infinite paths on $E$ together with the set of all finite paths starting at a source. 
We say that $E$ is {\it cofinal} if for each $\mu \in E^{\le \infty}$ and for each $w\in E^0$ there exists a (finite) path $\gamma $ such that $s(\gamma)$ is a vertex in the path $\mu$ and $r(\gamma )= w$. 

For vertices $v,w$ in $E^0$ write $w\le v$ if there exists a path $\mu$ such that $s(\mu) = v$ and $r(\mu) =w$. This defines a preorder relation on $E^0$.

\begin{theorem}
	\label{thm:charac-simplicity}
	Let $E$ be a row-finite directed graph. Then $C^*(E)$ is simple if and only if $E$ is cofinal and each cycle of $E$ has an entry. 
\end{theorem}

{\it Sketch of the proof.} Suppose that $E$ is cofinal and each cycle has an entry. We show that if $\{ S,P\}$ is a nonzero Cuntz-Krieger $E$-family in a $C^*$-algebra $A$, then $P_w\ne 0$ for all $w\in E^0$. By assumption there is a vertex $v$ in $E^0$ such that $P_v\ne 0$.  

Let $w\in E^0$. If $v$ is not a source, the relation (CK2) implies that there is an edge $e$ such that $r(e)= v$ and $S_eS_e^*\ne 0$. Then  $P_{s(e)}=S_e^*S_e \ne 0$. If $s(e)$ is not a source, the same argument can be repeated at $s(e)$. In this way we build either an infinite path $\mu$ or a finite path $\mu$ starting at a source such that $P_u\ne 0$ for all vertex $u$ in $\mu$. By cofinality there is a finite path $\alpha$ such that $s(\alpha)= u$ and $r(\alpha) = w$, where $u$ is a vertex in $\mu$. Since $S_\alpha^* S_{\alpha}=P_u\ne 0$ and
$P_wS_{\alpha}S_{\alpha}^* = S_\alpha S_\alpha ^* \ne 0$, it follows that $P_w\ne 0$. 

This implies that $P_w\ne 0$ for all $w\in E^0$. Hence any nonzero representation $C^*(E)\to A$ is injective by Theorem \ref{thm:CKuniqueness}, and it follows that $C^*(E)$ is simple.

For the converse, we only show that $E$ must be cofinal if $C^*(E)$ is simple.  Indeed if $\mu \in E^{\le \infty}$, and the set of vertices $w$ such that there is no finite path from a vertex of $\mu$ to $w$ is non-empty, then the set $H$ of these vertices would generate a nontrivial closed two-sided ideal $I$ of $C^*(E)$, contradicting the simplicity of $C^*(E)$. Indeed it is trivial that 
$H\ne \emptyset$ and that $H\ne E^0$, hence we are done if we can show that 
$$H= \{v\in E^0 : p_v\in I\}$$
because then necessarily $I\ne 0$ and $I\ne C^*(E)$. 
 To show the above set equality, one needs to verify that $H$ satisfies two important properties in the theory. First, $H$ is a {\it hereditary subset} of $E^0$, meaning that whenever $v\le w$ in $E^0$ and $v\in H$ then also $w\in H$. Second, $H$ is a {\it saturated subset} of $E^0$, meaning that whenever $v\in E^0$ is not a source in $E$ and $s(e)\in H$ for all $e\in r^{-1}(v)$, then $v\in H$. It is easy to check that $H$ satisfies both conditions.

We refer the reader to \cite[Theorem 4.14]{rae} for the detailed proof. 
\qed

\medskip

We obtain that the Cuntz algebra $\mathcal O _n= C^*(R_n)$ (with $n\ge 2$) is simple. However $C^*(R_1) = C(\mathbb T)$ is not simple since the only cycle in $R_1$ does not have an entry.  

\subsection{The monoid of projections}

The  Murray-von Neumann monoid of projections $\V(A)$ of a $C^*$-algebra $A$ is the monoid of Murray-von Neumann equivalence classes of projections $p$ in $M_\infty (A)$, endowed with the operation of orthogonal sum. 

For a ring $R$, the monoid $\V (R)$ is defined as the monoid of equivalence classes of idempotents in $M_{\infty}(R)$, by the relation $e\sim f$ if and only if there are $x,y\in M_{\infty}(R)$ such that $e= xy$ and $f=yx$. The monoid $\V(R)$ can be identified with the monoid of isomorphism classes of finitely generated projective right $R$-modules. In the non-unital case, this must be interpreted as the monoid of isomorphism classes of finitely generated projective $S$-modules $P$ such that $P=PR$, where $S$ is a (fixed) unital ring containing $R$ as a two-sided ideal, see e.g. \cite[Section 5]{Kenlimits}. For a $C^*$-algebra $A$ the two definitions of $\V (A)$ agree by \cite[Section 4.6]{Black}.

For a row-finite graph $E$, one may define a (commutative) monoid $M(E)$ associated to $E$. Concretely, $M(E)$ is the commutative monoid with a family of generators $a_v$, for $v\in E^0$, and the defining relations
\begin{equation*}
	a_v = \sum _{e\in r^{-1}(v)} a_{s(e)}  \qquad (v\in E^0\setminus \Sour (E))
\end{equation*} 

Observe that we always have a well-defined monoid homomorphism
$$\eta_E\colon M(E)\to \V (C^*(E))$$ defined by $\eta_E (a_v) = [p_v]\in \V (C^*(E))$. 

\begin{theorem}{\rm \cite{amp07}}
	\label{thm:APM}
	Let $E$ be a row-finite graph. Then the map $\eta_E$ is a monoid isomorphism.   
\end{theorem}

The proof goes through the {\it Leavitt path algebra} $L(E)$ of $E$. Namely, it is first shown in \cite[Theorem 3.5]{amp07} that the natural map $M(E)\to \V (L(E))$ is an isomorphism, and then it is shown in \cite[Theorem 7.1]{amp07} that the natural map $L(E)\to C^*(E)$ induces an isomorphims of $V$-monoids
$V(L(E))\cong \V(C^*(E))$. Obviously the composition of these two monoid isomorphisms is the canonical map
$\eta_E$, which is therefore an isomorphism.  

As a consequence one can get a formula for $K_0(C^*(E))$, although this can also be obtained in a different way \cite[Chapter 7]{rae}. We denote by $A'_E$ the matrix obtained from the adjacency matrix $A_E$ of $E$ by suppressing all the zero rows (the ones corresponding to the sources of $E$). This will be called the {\it reduced} adjacency matrix of $E$. 
We consider $(A_E')^T$ as a map from $\Z^{(E^0\setminus \Sour (E))}$ to $\Z^{(E^0)}$. 

\begin{theorem}{\rm \cite[Chapter 7]{rae}}
	\label{thm:KtheoryofgraphCs}
	Let $E$ be a row-finite graph, and let $A'_E$ be the reduced adjacency matrix of $E$. Then we have
	\begin{enumerate}
		\item[(a)] $K_0 (C^*(E))\cong \coker (I-(A_E')^T)$.
		\item[(b)] $K_1(C^*(E))	\cong \ker (I-(A_E')^T)$. 
	\end{enumerate}	
\end{theorem}

For instance for the $n$-rose quiver, we have $K_0(\mathcal O _n)\cong \Z/(n-1)\Z$
and $K_1(\mathcal O _n)= 0$. Observe that $K_1( C^*(E) )$ is always a subgroup of a free abelian group, so it is itself a free abelian group. Szymański \cite{szy} proved that every pair $(G_0,G_1)$ consisting of a countable abelian group $G_0$ and a free countable abelian group $G_1$ can be realized as the pair $(K_0(C^*(E)),K_1(C^*(E)))$ for some row-finite graph $E$ such that $C^*(E)$ is a simple and purely infinite $C^*$-algebra. To obtain all possible pairs $(G_0,G_1)$ of countable abelian groups, one can use Katsura algebras, which have also a graphical model, found by Exel and Pardo, see Section \ref{sect:ExelPardo}.

\subsection{The groupoid model}

We now introduce the groupoid picture for graph $C^*$-algebras. These groupoids have been introduced in a certain {\it ad hoc} way in several instances, and this is indeed the case in the initial paper \cite{kprr97}. However the construction of the groupoid can be done from an algebraic structure called an {\it inverse semigroup}, see \cite{Exel08} and \cite{ExelBook} for this approach. This is the most satisfying way to build this structure, but we will restrict to the ad hoc method in these notes. 

We refer the reader to \cite{SimsCRMNotes} for a quick introduction to the theory of étale groupoids. 

A groupoid is a small category with inverses. Although this is the most economical definition, we want to phrase it in more concrete terms. We will deal thus with a set $\mathcal G$, endowed with two maps
$s,r\colon \mathcal G \to \mathcal G^0$, where $\mathcal G^0$ is the set of objects of $\mathcal G$. These maps are called the source and the range of $\mathcal G$. We have a partially defined multiplication 
$$\mu\colon \mathcal G ^{(2)}\to \mathcal G, \qquad \mu (g,h)= gh$$
where $\mathcal G ^{(2)}$ is the set of pairs $(g,h)\in \mathcal G\times \mathcal G$ such that $s(g)= r(h)$, and $s(gh)= s(h)$, $r(gh) = r(g)$. Moreover the product is associative and each ``arrow" $g\in \mathcal G$
is invertible, with $s(g^{-1})= r(g)$, $r(g^{-1})= s(g)$, $gg^{-1}= r(g)$, $g^{-1}g = s(g)$.  

We consider only étale topological groupois, which are groupoids $\mathcal G$ with a locally compact topology
such that the group operations are continuous and the source and range maps $s,r\colon \mathcal G \to \mathcal G^0$ are local homeomorphisms. The étale condition implies that $\mathcal G^0$ is an open subset of $\mathcal G$ \cite[Lemma 8.4.2]{SimsCRMNotes}. 

It is always assumed that $\mathcal G^0$ is a (locally compact) Hausdorff space, but it is not assumed that the groupoid itself is Hausdorff. Indeed in some interesting situations, the groupoid $\mathcal G$ is not Hausdorff. 

An étale groupoid is said to be {\it ample} if it has a basis of compact open sets. All groupoids appearing in these notes will be ample groupoids. 

Given a Hausdorff ample groupoid $\mathcal G$, one can define a convolution product and an involution on the algebra $C_c(\mathcal G)$ of compactly supported continuous complex functions:
$$(f \ast g ) (x)= \sum_{yz=x} f(y)g(z),\qquad (f^*)(x) = \ol{f(x^{-1})}.$$ The universal $C^*$-algebra associated to $C_c(\mathcal G)$ is called the groupoid $C^*$-algebra of $\mathcal G$, and it is denoted by $C^*(\mathcal G)$. A reduced $C^*$-algebra $C_r^* (\mathcal G)$ can also be defined, see \cite{SimsCRMNotes}.

\begin{example}
	\label{exam:transformation-groupoid}
	Let $G$ be a discrete group acting by homeomorphisms on a compact set $X$. Then the {\it transformation groupoid} associated to the action is the groupoid 
	$\mathcal G _{G,X}$ whose underlying set is  $G\times X$, with unit space $\{e\}\times X\cong X$, and with  $s(g,x) = x$, $r(g,x)= gx$, and
	$$(h,gx) (g,x) = (hg,x),\qquad (g,x)^{-1} = (g^{-1}, gx) .$$
\end{example}

The groupoid $C^*$-algebra associated to the transformation groupoid of an action of $G$ on $X$ as above is isomorphic with an important structure in the theory of $C^*$-algebras, called the {\it crossed product}.
First we observe that if $\alpha$ is an action of $G$ on $X$ as above, then $\alpha$ induces an action of $G$ on $C(X)$ by $*$-automorphisms by $\alpha_\gamma (f)(x)= f(\alpha_\gamma^{-1}(x))$ for $\gamma\in G$, $f\in C(X)$ and $x\in X$. 
Now  there is an {\it algebraic} crossed product algebra $C(X)\alrtimes_\alpha G$, which is the linear span of elements of the form $f\delta_\gamma$, where $f\in C(X)$ and $\gamma \in G$, with the sum defined componentwise and the product and involution defined by
\begin{equation}
	\label{eq:crossedprod-operations}
	(f\delta_\gamma )(g\delta_\mu) = (f\alpha_\gamma (g)) (\delta _{\gamma \mu}),\qquad (f\delta_\gamma)^* = \alpha_{\gamma^{-1}} (f^*)\delta_{\gamma^{-1}}.
\end{equation}
The crossed product $C(X)\rtimes_{\alpha} G$ is then the enveloping $C^*$-algebra of the algebraic crossed product  $C(X)\alrtimes _{\alpha} G$. 

\begin{theorem}
	\label{thm:cross-product}
	Let $G$ be a discrete group acting by homeomorphisms on a compact set $X$. Then 
	$$C^*(\mathcal G_{G,X}) \cong C(X)\rtimes_\alpha G.$$
\end{theorem}

Thinking of the algebraic crossed product $C(X)\alrtimes G$ as the algebra $C_c(G, C(X))$, the isomorphism in Theorem \ref{thm:cross-product} is induced by an isomorphism $\omega\colon C_c (\mathcal G_{G,X})\to C_c(G,C(X))$ given by
$$\omega (f)(\gamma) (x)= f(\gamma,x),\qquad (f\in C_c(\mathcal G), \gamma \in G, x\in X)\, ,$$
see \cite[Example 9.1.7]{SimsCRMNotes}. 

We can now give the definition of the groupoid $\mathcal G _E$ associated to a graph $E$. We will assume that $E$ has no sources. Let $E^{\infty}$ be the set of all the infinite paths $e_1e_2\cdots $ on $E$. Given $x= e_1e_2\cdots \in E^{\infty}$ define
$$\sigma (x) =  \sigma (e_1e_2e_3 \cdots ) = e_2e_3\cdots \in E^{\infty}.$$ 

\begin{definition}
	\label{def:graph-groupoid}
	Let $E$ be a row-finite graph without sources.  The graph groupoid $\mathcal G_E$ is given by
	$$\mathcal G_E = \{(x,k-l,y): x,y\in E^{\infty}, k,l\in \Z^+, \sigma^k(x) = \sigma^l (y)\}.$$
	The triples $(x,n,y)$ and $(y',m,z)$ are composable if and only if $y=y'$ and in this case
	$$(x,n,y)(y,m,z) = (x,n+m,z).$$
	The inverse of $(x,n,y)$ is $(x,n,y)^{-1} = (y,-n,x)$. The space of units is $E^{\infty}$ identified with the set of triples $(x,0,x)$.
	
	A basis of compact open subsets of $E^\infty$ is given by all the {\it cylinder sets} $\mathcal Z (\mu)$, where $\mu$ is a finite path and 
	$$\mathcal Z (\mu) = \{ x\in E^{\infty} : x= \mu y \text{ for some } y\in E^\infty\}$$
	is the set of all infinite paths ending in $\mu$.

	The topology of $\mathcal G_E$ is generated by the cylinder sets $\mathcal Z (\mu,\nu)$ given by
	$$\mathcal Z (\mu,\nu) = \{ (x,|\mu| - |\nu|,y) \in \mathcal G_E : x\in \mathcal Z (\mu), y\in \mathcal Z (\nu)\}.$$
\end{definition}

\begin{theorem}
	\label{thm:isomorphism-groupoid-graph}
	Let $E$ be a row-finite graph without sources. Then there is an isomorphism
	$$C^*(E)  \to C^*(\mathcal G_E)$$
	sending $p_v$ to $1_{\mathcal Z (v)}$ and $s_e$ to $1_{\mathcal Z (e,s(e))}$. 
	The graph groupoid $\mathcal G_E$ is always a Hausdorff groupoid. 
\end{theorem}

\begin{example}
	The {\it Cuntz groupoid} $\mathcal G_n$ is the groupoid associated to the Cuntz algebra $\mathcal O_n$.
	Set $E=R_n$. Then $E^\infty = \{x_1,\dots , x_n\}^{\N}$, with the product topology, and the generators $x_i$ of $\mathcal O_n$ are identified with the characteristic functions of the cylinder sets 
	$$\mathcal Z_i = \{(x_i \gamma , 1,  \gamma ) :  \gamma \in E^\infty\}.$$
	Observe that 
	$$1_{\mathcal Z_i}^*1_{\mathcal Z_j} = 1_{\mathcal Z (v,-1,x_i)}1_{\mathcal Z (x_j,1,v)} = \delta_{i,j} 1_{\mathcal Z(v)} = 1$$
	and 
	$$\sum_{i=1}^n 1_{\mathcal Z_i}1_{\mathcal Z_i}^* = \sum_{i=1}^n 1_{\mathcal Z (x_i)} = 1.$$  
\end{example}

Recall that a {\it bisection} of a groupoid $\mathcal G$ is a subset $B$ of $\mathcal G $ such that the restrictions of both $s$ and $r$ to $B$ are injective.

One can introduce a {\it type semigroup} associated to an ample Hausdorff groupoid $\mathcal G$. This is the commutative semigroup $\Typ (\mathcal G)$ generated by classes $[U]$ of open compact subsets $U$ of $\mathcal G^0$ subject to the relations:
\begin{enumerate}
	\item $ [\emptyset] = 0$, 
	\item $[U\cup V] = [U]+ [V]$ if $U\cap V=\emptyset$,
	\item $[r(U)]= [s(U)]$ if $U$ is an open compact bisection of $\mathcal G$. 
\end{enumerate}

There is a well-defined monoid homomorphism
$$\mathfrak{t}_{\mathcal G} \colon \Typ (\mathcal G) \to \V (C^*(\mathcal G))$$
defined by $\mathfrak{t}_{\mathcal G} ([U]) = [1_U]$ for each open compact subset $U$ of $\mathcal G^0$.
To see that it is well-defined one has to check the preservation of (1)-(3) in the definition of the type semigroup. (1) and (2) are clearly preserved. For (3), suppose that $U$ is an open compact bisection of $\mathcal G$. Since $1_A1_B= 1_{AB}$ for every pair of open compact bisections $A,B$ of $\mathcal G$, we have
$$1_U1_U^* = 1_{UU^{-1}}= 1_{r(U)},\qquad 1_U^* 1_U = 1_{U^{-1}U} = 1_{s(U)}.$$
Hence $1_{r(U)}\sim 1_{s(U)}$ and thus (3) is preserved as well. 

It is shown in \cite{KMP25} that $\Typ (\mathcal G_E )\cong M(E)$ for all row-finite graphs $E$, with $ [\mathcal Z (v)]$ corresponding to $a_v$ for each $v\in E^0$. See also \eqref{eq:type-semigroup-ssactions} below. Combining this fact with Theorem \ref{thm:APM} one obtains that the natural map
$$ \mathfrak{t}_{\mathcal G_E} \colon \Typ (\mathcal G_E) \to \V (C^*(E))$$
is an isomorphism for each row-finite graph without sources $E$.

\section{Exel-Pardo $C^*$-algebras}
\label{sect:ExelPardo}

We will follow \cite{ep17} to introduce Exel-Pardo $C^*$-algebras. These algebras generalize the $C^*$-algebras associated to self-similar groups \cite{Nekra}, and also the $C^*$-algebras introduced by Katsura in \cite{kat08}.

We refer the reader to \cite{Pardo} for a recent overview on self-similar graphs and their associated algebras. 

Let $G$ be a countable discrete group and let $E$ be a finite graph without sources. 

Let $\sigma \colon G \to \Aut (E)$ be an action of $G$ on $E$, that is, $g\mapsto \sigma_g$ is a group homomorphism from $G$ into the group of graph automorphisms of $E$. Hence $\sigma_g (s(e))= s(\sigma_g(e))$ and $\sigma_g (r(e)) = r(\sigma_g (e))$ for all $g\in G$ and $e\in E^1$.

Let $\varphi\colon G\times E^1\to G$ be a one-cocycle for $\sigma$, that is,
\begin{equation}
	\label{eq:1-cocycle}
	\varphi (gh,e) = \varphi (g,\sigma_h(e))\varphi (h,e)
\end{equation}
for all $g,h\in G$ and all $e\in E^1$, which additionally satisfies
\begin{equation}
	\label{eq:additiona-self-sim}
	\sigma_{\varphi (g,e)}(s(e)) = \sigma _g (s(e)) \qquad \forall g\in G,\quad \forall e\in E^1. 
\end{equation}


\begin{remark}
	\label{rem:forEP-relations}
	By setting $g=h=1$ in \eqref{eq:1-cocycle} one gets 
	\begin{equation}
		\label{eq:one-in-one-cocycle}
		\varphi (1,e)= 1 \qquad \forall e\in E^1.
	\end{equation}
\end{remark}

The pair $(\sigma, \varphi)$ may be extended to a corresponding pair $(\sigma^*,\varphi^*)$ on the set $E^*$ of finite paths on $E$. This extension satisfies the following key relations (see \cite[Proposition 2.4]{ep17}):
\begin{equation}
	\label{eq:key-for-extended1}
	\sigma_g^* (\alpha\beta) = \sigma_g^* (\alpha)\sigma^* _{\varphi^* (g,\alpha)}(\beta) 
\end{equation}
\begin{equation}
	\label{eq:key-for-extended2}
	\varphi^* (g,\alpha\beta) = \varphi^* (\varphi^* (g,\alpha),\beta).
\end{equation}

From now on, we denote by $g\alpha$ the action $\sigma^*_g(\alpha)$ for a path $\alpha$. We also write $g|_{\alpha}$ instead of $\varphi^* (g,\alpha)$. Relations \eqref{eq:key-for-extended1} and \eqref{eq:key-for-extended2} become:
\begin{equation}
	\label{eq:key-for-extendedNewNot1}
	g (\alpha\beta) = (g \alpha) (g|_\alpha \beta) 
\end{equation}
\begin{equation}
	\label{eq:key-for-extendedNewNot2}
 g|_{\alpha\beta} = (g|_\alpha)|_\beta.
\end{equation}
The action of $G$ on $E$ can also be extended to the space $E^{\infty}$ of infinite paths on $E$. 

\begin{definition}
	The $C^*$-algebra $\mathcal  O _{G,E}$ associated to the 
	self-similar graph $(G,E,\varphi)$ is  the universal $C^*$-algebra generated by a set   
	$$\{ p_v : v\in E^0 \}\cup \{ s_e : e\in E^1\} \cup \{ u_g : g\in G\}$$
	such that
	\begin{enumerate}
		\item $\{ p_v : v\in E^0 \}\cup \{ s_e : e\in E^1\}$ is a Cuntz-Krieger $E$-family,
		\item $u\colon G\to \mathcal O _{G,E}$, $g\mapsto u_g$, is a unitary representation of $G$,
		\item $u_gs_e = s_{ge} u_{g|_e}$ for all $g\in G$ and all $e\in E^1$,
		\item $ u_gp_v = p_{gv} u_g$ for all $g\in G$ and $v\in E^0$. 
	\end{enumerate} 
\end{definition}

The $C^*$-algebra $\mathcal O_{G,E}$ is the closed linear span of the set $\mathcal S$ of elements of the form
$$s_\alpha u_g s_\beta^* $$
such that $\alpha,\beta \in E^*$, $g\in G$ and $s(\alpha) = gs(\beta)$, see \cite[Proposition 3.9]{ep17}.

\begin{definition}
	\label{def:pseudo-free}
	We say that $(G,E,\varphi)$ is {\it pseudo-free} if whenever $g\in G$, $e\in E^1$, $ge= e$ and $\varphi (g,e)= 1$ we have $g=1$.  
\end{definition}

The natural map $C^*(E)\to \mathcal O _{G,E}$ is always injective, and the natural map $C^*(G)\to \mathcal O _{G,E}$ is injective whenever $(G,E)$ is pseudo-free \cite[Propositions 11.1 and  11.8]{ep17}.

Conditions for simplicity of $\mathcal O _{G,E}$ are given in \cite[Theorem 16.1]{ep17}. Very roughly, these conditions are $G$-equivariant versions of the conditions for simplicity of $C^*(E)$, given in  Theorem \ref{thm:charac-simplicity}.

An interesting fact is that one can describe an ample groupoid $\mathcal G_{G,E}$ so that
$$\mathcal O _{G,E}\cong C^*(\mathcal G_{G,E}).$$
The groupoid $\mathcal G_{G,E}$ resembles the groupoid $\mathcal G_E$ associated to a directed graph, but of course incorporating the group $G$ in the picture. We refer the reader to \cite{ep17} for full details of the construction.  

Let us give here a description of the groupoid $\mathcal G_{G,E}$ in the case where $(G,E)$ is pseudo-free, following \cite[Proposition 8.6]{ep17}. First define a set $\mathcal S_{G,E}$ by
$$\mathcal S_{G,E} = \{ (\alpha, g ,\beta) \in E^*\times G\times E^* : s(\alpha)= g \, s(\beta)\}.$$
The elements of $\mathcal S_{G,E}$ are meant to model the elements $s_{\alpha}u_gs_{\beta}^*$ of the algebra $\mathcal O_{G,E}$. As such, a product and an involution can be defined on $\mathcal S_{G,E}$ that make it an {\it inverse semigroup}, see \cite{ep17}.

Now define an equivalence relation $\sim$ on the set of elements $(\alpha,g, \beta;\xi)$, where $(\alpha,g,\beta)\in \mathcal S_{G,E}$ and $\xi \in E^{\infty}$ is such that $\xi \in \mathcal Z (\beta)$ as follows. Let $(\alpha_i,g_i,\beta_i; \eta_i)$, $i=1,2$, two quadruples as before, and assume that $|\beta_1| \le |\beta_2|$. Then $(\alpha_1,g_1,\beta_1; \eta_1)\sim (\alpha_2,g_2,\beta_2; \eta_2)$ if and only if there exists a finite path $\gamma $ and an infinite path $\xi$ such that
\begin{enumerate}
	\item $\alpha_2= \alpha _1 g_1 \gamma $,
	\item $g_2 = g_1|_{\gamma}$,
	\item $\beta_2 = \beta_1 \gamma $,
	\item $\eta_1 = \eta_2 = \beta_1 \gamma \xi $. 
\end{enumerate}

The set of equivalence classes, denoted by $[\alpha,g,\beta; \xi]$, can be endowed with a structure of ample groupoid,
which is denoted by $\mathcal G _{G,E}$ \cite{ep17}. We have $\mathcal G_{G,E}^0 = E^{\infty}$, and
$$s([\alpha, g,\beta; \beta \gamma]) = \beta \gamma , \qquad r([\alpha,g,\beta; \beta \gamma]) = \alpha g \gamma .$$
Since we are assuming that $(G,E)$ is pseudo-free, the groupoid $\mathcal G_{G,E}$ is Hausdorff \cite[Proposition 12.1]{ep17}. However there are important cases where the associated groupoid $\mathcal G_{G,E}$, constructed in \cite{ep17}, is not Hausdorff (and thus the self-similar action is not pseudo-free). 
This is the case for the famous Grigorchuk self-similar group, which was the first finitely generated group of intermediate growth. In \cite{CEPSS}, the authors make an extensive study of the $C^*$-algebra $\mathcal O _{G,E}$ associated to the Grigorchuk self-similar group $(G,E)$.

The following are key examples of the theory.

\begin{example}
	\label{exam:self-similar-group}
	Let $G$ be a group and $X$ a finite set. A  faithful (level preserving) action of $G$ on the set $X^*$ of finite words on $X$ is a {\it self-similar action} if for each $g\in G$ and $x\in X$ there is
	$g|_x\in G$ such that $g(xw) = g(x)g|_x (w)$ for all $w\in X^*$. This can be interpreted as a self-similar action of $G$ on the $n$-rose graph $R_n$.  The algebras associated to self-similar actions have been studied by Nekrashevych \cite{Nekra}, amongst others. 
\end{example}

\begin{example}
	\label{exam:Katsura-algebras} Fix a positive integer $N$. Let $A\in M_{N\times N} (\Z^+)$ be a non-negative square matrix, and let $B\in M_{N\times N}(\Z)$ be a square matrix. We will assume that $A$ has no zero rows and that $A_{ij}=0\implies B_{ij}=0$. Let $E$ be the directed graph with $E^0=\{1,\dots ,N\}$ and with $A_{ij}$ edges from vertex $j$ to vertex $i$. We fix the notation 
	$$\{ e_{ijn} : 1\le i,j\le N, 0\le n < A_{ij}\}$$
	for the set of edges from vertex $j$ to vertex $i$.
	
	We define an action $\sigma$ of $\mathbb Z$ on $E$ as follows. The action of $\mathbb Z$ on $E^0$ is trivial. Let $m\in \Z$, and let $e_{ijn}$ be an edge of $E$. Performing the euclidean division of $n+mB_{ij}$ by $A_{ij}$ we get unique integers $\hat{k},\hat{n}$ such that  
	$$n+mB_{ij} = \hat{k}A_{ij} + \hat{n},\qquad 0\le \hat{n} <A_{ij}.$$
	Then we set $\sigma_m (e_{ijn}) = e_{ij\hat{n}}$ and $\varphi (m,e_{ijn}) = \hat{k}$. The corresponding $C^*$-algebra $\mathcal O_{\Z,E}$ coincides with Katsura algebra $\mathcal O _{A,B}$ built in \cite{kat08}. 
\end{example}

Katsura algebras (allowing $N$ to be $\omega$), realize all Kirchberg algebras (separable, nuclear, simple, purely infinite and with UCT) by a result of Katsura \cite{kat08}. Indeed, Katsura obtained in \cite{kat08} the following formulas for $K_0$ and $K_1$:
$$K_0 (\mathcal O_{A,B}) \cong \coker (I-A^T) \oplus \ker (I-B^T),$$
$$K_1 (\mathcal O_{A,B}) \cong \coker (I-B^T) \oplus \ker (I-A^T).$$
See also \cite{MS}.

Using Katsura algebras, we will see an example where the canonical map $\mathfrak{t}_{\mathcal G} \colon \Typ (\mathcal G)\to \V (C^*(\mathcal G))$ defined in Section \ref{sect:graphs} is not surjective.
First we notice a recent result of Kwaśniewski, Meyer and Prasad \cite{KMP25}, concerning the type semigroup $\Typ (\mathcal G_{G,E})$ associated to a self-similar graph $(G,E)$. In \cite[Corollary 7.7]{KMP25} the authors show that the type semigroup of the groupoid 
$\mathcal G _{G,E}$ is precisely the monoid of coinvariants $M(E)_G$ of the graph monoid $M(E)$ of $E$. 
First observe that $G$ acts on $M(E)$ by $g\cdot a_v = a_{gv}$. To see that the action is well-defined, we need to check that this rule defines a monoid homomorphism of $M(E)$ into itself. For this purpose, take a vertex $v\in E^0$ and $g\in G$. Then we need to check that the relation
$$a_{gv} = \sum_{e\in r^{-1}(v)} a_{gs(e)}$$
holds in $M(E)$. However since $g$ acts by automorphisms on $E$, we have that 
$r^{-1} (gv) = \{ ge: e\in r^{-1}(v)\}$ and thus
$$\sum_{e\in r^{-1}(v)} a_{gs(e)} = \sum_{f\in r^{-1}(gv)} a_{s(f)} = a_{gv},$$
as desired. We can now consider the monoid of coinvariants $M(E)_G$ for this action, which is the quotient of $M(E)$ by the congruence generated by all the pairs $(z, g\cdot z)$ for all $z\in M(E)$. 

In the following lemma, we find a convenient presentation of the monoid $M(E)_G$, which is indeed a graph monoid. The graph $E_G$ appearing in Lemma \ref{lem:coinvariants} has been first constructed by Larki in \cite{Larki}.

\begin{lemma}
	\label{lem:coinvariants}
	Let $\Omega$ be a complete set of representatives for the $G$-action on $E^0$.
	Consider the graph $E_G$ such that $E_G^0=E^0/G$ and $E_G^1 = \bigsqcup_{v\in \Omega} r^{-1} (v)$, with $r_G(e)=[r(e)]$ and $s_G(e) = [s(e)]$ for all $e\in E_G^1$. Then $M(E)_G\cong M(E_G)$ and therefore $M(E)_G$ is a graph monoid.
\end{lemma}

\begin{proof}  
	Define $\phi\colon M(E_G)\to M(E)_G$ by $\phi (a_{[v]}) = [a_v]\in M(E)_G$. The relations of $M(E_G)$ are clearly preserved. Indeed, if $v\in E^0$ and $\ol{v}\in \Omega$ is the unique element of $\Omega$ such that $[v]=[\ol{v}]$ in $E^0/G$, then we have
	$$\phi (a_{[v]}) = [a_v] = [a_{\ol{v}}] = [\sum _{e\in r^{-1}(\ol{v})} a_{s(e)}] = \sum_{e\in r_G^{-1}([v])} [a_{s(e)}] = \sum _{e\in r_G^{-1}([v])} \phi (s_G (e)).$$
	Hence $\phi$ defines a monoid homomorphism. Conversely, define $\psi \colon M(E)\to M(E_G)$
	by $\psi (a_v)= a_{[v]}$. The relations of $M(E)$ are preserved because every  $g\in G$ induces a bijection between $r^{-1}(v)$ and $r^{-1}(gv)$, for all $v\in E^0$. To see that $\psi$ factors through $M(E)_G$, it suffices to check that $\psi (a_v) = \psi (g\cdot a_v)$. But this is obvious, because $g\cdot a_v = a_{gv}$. 
    Hence we obtain a homomorphism $\ol{\psi}\colon M(E)_G\to M(E_G)$. 
    The maps $\phi$ and $\ol{\psi}$ are clearly mutually inverse.   
\end{proof}

Hence, we can state, by \cite[Proposition 7.7]{KMP25}
\begin{equation}
	\label{eq:type-semigroup-ssactions}
	\Typ (\mathcal G_{G,E}) \cong M(E)_G\cong M(E_G).
\end{equation} 
In particular, taking as $G$ the trivial group, we obtain that $\Typ (\mathcal G _E) \cong M(E)$, providing a proof of the fact mentioned in Section \ref{sect:graphs}.

Let us just show that there is a natural well-defined homomorphism $\lambda \colon M(E)_G\to \Typ (\mathcal G)$, where $\mathcal G = \mathcal G_{G,E}$. Denoting by $[v]$ the equivalence classes of $v\in E^0$ in $E^0/G$, we define 
$$\lambda (a_{[v]}) = [\mathcal Z (v)].$$
Now notice that there is a basis of open compact subsets of $\mathcal G$ of the form $\mathcal Z (\alpha,g,\beta)$, for $(\alpha,g,\beta)\in \mathcal S_{G,E}$. 
Taking $U=\mathcal Z (gv,g,v)$, for $g\in G$, we see that $UU^*= \mathcal Z (gv)$ and $U^*U=\mathcal Z (v)$.
This shows that $[\mathcal Z (v)]$ does not depend on the representative of $[v]$. 
Now one easily shows that, for $v\in E^0$, 
$$[\mathcal Z (v)] = \sum _{e\in r^{-1}(v)} [\mathcal Z (s(e))],$$
using the open compact bisections $V_e= \mathcal Z (e, 1, s(e))$ for $e\in r^{-1}(v)$. The homomorphism $\lambda$ is shown to be an isomorphism in \cite[Proposition 7.7]{KMP25}.

\medskip

We are now ready to present an example where the canonical map $\mathfrak{t}_{\mathcal G} \colon \Typ (\mathcal G)\to \V (C^*(\mathcal G))$ is not surjective. Observe that if in Example \ref{exam:Katsura-algebras} we take $B$ such that $\ker (I-B^T)\ne 0$,  then the above map cannot be surjective, since the image of $\Typ  (\mathcal G_{A,B}) \cong M(E)_G$ in $K_0(C^*(\mathcal G_{A,B}))$ is contained in the factor $\coker (I-A^T)$. Observe that for Katsura algebras we have that $M(E)_G=M(E)$, since $G$ fixes the vertices of $E$.

\begin{example}
	\label{exam:lamplighter}
	We now present an example of a self-similar action of the lamplighter group. 
	This group is important because it gave the first counterexample to the Strong Atiyah Conjecture (SAC), see \cite{GZ} and also \cite{DiSc, Gra, AC21}. Indeed the approach using the self-similar action was essential for the computations of the spectra of random walks on this group, given by Grigorchuk and  Żuk in \cite{GZ}. 
	
	\begin{definition}
    \label{def:lamplighter}
		The \emph{lamplighter group} $\Gamma$ is defined to be the wreath product of the finite group of two elements, $\Z_2$, by $\Z$. In other words,
		$$\Gamma = \Z_2 \wr \Z = \Big( \bigoplus_{i \in \Z} \Z_2 \Big) \rtimes_{\sigma} \Z,$$
		where the action implementing the semidirect product is the well-known Bernoulli shift $\sigma$ defined by
		$$\sigma_n(x)_i = x_{i+n} \quad \text{ for }x = (x_i) \in \bigoplus_{i \in \Z} \Z_2.$$
		If we denote by $t$ the generator corresponding to $\Z$ and by $s_i$ the generator corresponding to the $i^{\text{th}}$ copy of $\Z_2$, we have the following presentation for $\Gamma$:
		$$\Gamma = \langle t,\{s_i\}_{i \in \Z} \mid s_i^2 , s_is_js_is_j, ts_it^{-1}s_{i-1} \text{ for } i,j \in \Z \rangle.$$
	\end{definition}

	Grigorchuk, Nekrashevich and Sushchanskii \cite{GNS00} have shown that the lamplighter group appears as the group of automorphisms of $\Z_2[[x]]$ generated by addition by $1$, denoted by $\alpha_1$, and by multiplication $\mu_f$ by $f$, where $f(x)=(1-x)^{-1}$. The generators giving the self-similar presentation are $a= \mu_f$ and $b=\mu_f\alpha_1$. The corresponding automaton is shown in Figure \ref{fig:lamplighter}. 	
	Skipper and Steinberg \cite{SS20} have generalized this construction by considering the algebra $R[[x]]	$, where $R$ is a finite commutative ring. The isomorphism between the self-similar group $L$ generated by $a$ and $b$ and the lamplighter group $\Gamma$ sends $a$ to $t$,  $a^{-1}b= \alpha_1$ to $s_0$, and in general
	$a^{-m}(a^{-1}b)a^m$ is sent to $s_m$ for all $m\in \Z$.  
\end{example}

\begin{figure}[htbp]
	\label{fig:lamplighter}
	\begin{center}
		\begin{tikzpicture}[->,shorten >=1pt,%
			auto,node distance=3cm,semithick,
			inner sep=5pt,bend angle=30]
			\tikzset{every loop/.style={min distance=10mm,looseness=10}}
			\node[state] (A)  {$a$};
			\node[state] (B) [right of=A] {$b$};
			\path 
			(A) edge [loop left] node [below]  {$0\mid 0$} (A)
			(A) edge [bend left] node [above]  {$1\mid 1$} (B)
			(B) edge [bend left] node [below]  {$1\mid 0$} (A)
			(B) edge [loop right] node [below]  {$0\mid 1$} (B);
		\end{tikzpicture}
	\end{center}
	\caption{Automaton for the lamplighter group}
\end{figure}
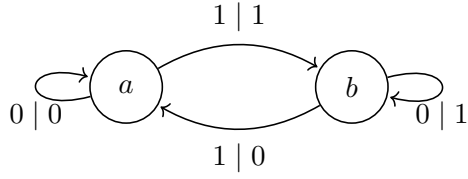

By \cite[Section 8]{MS}, the $C^*$-algebra $\mathcal O_{\Gamma,R_2}$ associated to the self-similar action of the lamplighter group $\Gamma$ is unital, nuclear, simple and purely infinite. Moreover $K_0(\mathcal O_{\Gamma,R_2}) =K_1(\mathcal O_{\Gamma,R_2}) =0$. Hence we can conclude from the classification theorem that $\mathcal O _{\Gamma,R_2} \cong \mathcal O_2$.

\begin{remark}
	\label{rem:HKConjecture} 
In \cite[Definition 5.4]{RS}, Rainone and Sims define the type semigroup of the ample groupoid $\mathcal G$ as the quotient $C_c (\mathcal G^{(0)}, \Z^+)/{\sim}$, where $f\sim g $ if and only if there are open compact bisections $U_1,\dots , U_n$ such that $f= \sum_{i=1}^n 1_{s(U_i)}$ and $g= \sum _{i=1}^n 1_{r(U_i)}$. They also observe that their type semigroup is isomorphic to the one introduced by B\"onicke and Li in \cite{bl20}, see \cite[Remark 5.5]{RS}. The fact that both semigroups coincide with the semigroup introduced in Section \ref{sect:graphs} was noticed in \cite[Proposition 7.3]{abps}.

By \cite{abbl23}, the map $\iota \colon \Typ (\mathcal G) \to H_0(\mathcal G)$ defined by $\iota ([f]) =[f]\in H_0(\mathcal G)$ is a well-defined semigroup homomorphism, where $H_0(\mathcal G)$ is the $0$-th homology group of the groupoid $\mathcal G$. Moreover $\iota (\Typ (\mathcal G)) = H_0(\mathcal G)^+$, and the semigroup homomorphism $\iota$ represents exactly the Grothendieck group of $\Typ (\mathcal G)$, see also \cite[Subsection 2.2]{Melleray} for related results. Note that there is a commutative diagram  
$$\xymatrix{\Typ (\mathcal G) \ar[d]_{\mathfrak t _{\mathcal G}} \ar[r]^{\iota}& H_0 (\mathcal G) \ar[d]^{\mathfrak h_{\mathcal G}}\\
	V(C^*(\mathcal G)) \ar[r] & K_0(C^* (\mathcal G))},$$
    where $\mathfrak h_{\mathcal G} \colon H_0 (\mathcal G)\to K_0(C^*(\mathcal G))$ and $V(C^*(\mathcal G))\to K_0(C^*(\mathcal G))$ are the standard homomorphisms. 

    Matui's Conjecture predicts in particular that the map $\mathfrak h_{\mathcal G} \colon H_0(\mathcal G)\to K_0(C^*(\mathcal G))$ is injective. Although this conjecture has been disproved in general, its validity has also been proved in many situations. By the diagram above, the semigroup homomorphism
    $\mathfrak t_{\mathcal G} \colon \Typ (\mathcal G)\to V(C^*(\mathcal G))$ can be considered as a non-stable version of the group homomorphism 
    $\mathfrak h_{\mathcal G} \colon H_0(\mathcal G)\to K_0(C^*(\mathcal G))$, and it is natural to investigate conditions under which it is injective or surjective.

\end{remark}

\section{Separated graph $C^*$-algebras}
\label{sect:separated}

Separated graph algebras were introduced in \cite{ag11} and \cite{ag12}. The purpose was to obtain a class of graph algebras whose structure of projections (finitely generated projective modules) is as general as possible. In particular, the algebras $L(m,n)$ considered by Leavitt \cite{Lea57, Lea62} are realized as full corners of the Leavitt path algebras associated to the separated graph $E(m,n)$. See Figure \ref{fig:m,nsepargraph} for the case of $E(2,3)$.
Similarly, the $C^*$-algebras built by McClanahan \cite{MCl92, MCl93} can be obtained as full corners of the graph $C^*$-algebras of the separated graphs $E(m,n)$. 

We will review here the definitions of these algebras. 

\begin{definition}{\rm \cite{ag12}} 
\label{defsepgraph}
	A \emph{separated graph} is a pair $(E,C)$ where $E$ is a graph,
	$C=\bigsqcup _{v\in E^ 0} C_v$, and $C_v$ is a partition of
	$r^{-1}(v)$ (into pairwise disjoint nonempty subsets) for every
	vertex $v$. (In case $v$ is a source, we take $C_v$ to be the empty
	family of subsets of $r^{-1}(v)$.)
	
	If all the sets in $C$ are finite, we say that $(E,C)$ is a
	\emph{finitely separated} graph. This necessarily holds if $E$ is
	row-finite (that is, if $r^{-1}(v)$ is a finite set for every
	$v\in E^0$.)
	
	The set $C$ is a \emph{trivial separation} of $E$ in case $C_v=
	\{r^{-1}(v)\}$ for each $v\in E^0\setminus \Sour (E)$. In that case,
	$(E,C)$ is called a \emph{trivially separated graph} or a
	\emph{non-separated graph}.
\end{definition}

The following definition gives the Leavitt path algebra $L(E,C)$
as a universal object in the category of complex $*$-algebras.

\begin{definition}
	\label{def:LPASG} The
	{\it Leavitt path algebra of the separated graph} $(E,C)$ is the complex $*$-algebra $L(E,C)$ with
	generators $\{ v, e\mid v\in E^0, e\in E^1 \}$, subject to the
	following relations:
	\begin{enumerate}
		\item[] (V)\ \ $vv^{\prime} = \delta_{v,v^{\prime}}v$ \ and \ $v=v^*$ \ for all $v,v^{\prime} \in E^0$ ,
		\item[] (E)\ \ $r(e)e=es(e)=e$ \ for all $e\in E^1$ ,
		\item[] (SCK1)\ \ $e^*e'=\delta _{e,e'}s(e)$ \ for all $e,e'\in X$, $X\in C$, and
		\item[] (SCK2)\ \ $v=\sum _{ e\in X }ee^*$ \ for every finite set $X\in C_v$, $v\in E^0$.
	\end{enumerate}
\end{definition}

The Leavitt path algebra $L(E)$ is just $L(E,C)$ where $C_v= \{
r^{-1}(v)\}$ if $r^{-1}(v)\ne \emptyset $ and $C_v=\emptyset $ if
$r^{-1}(v)=\emptyset$.

We now recall the definition of the graph C*-algebra $C^*(E,C)$,
introduced in \cite{ag12}.

\begin{definition} The \emph{graph C*-algebra} of a separated graph $(E,C)$ is the
	C*-algebra $C^*(E,C)$  with generators $\{ v, e \mid v\in E^0,\ e\in
	E^1 \}$, subject to the relations (V), (E), (SCK1), (SCK2). In other
	words, $C^*(E,C)$ is the enveloping C*-algebra of $L(E,C)$.
\end{definition}

In case $(E,C)$ is trivially separated, $C^*(E,C)$ is just the
classical graph C*-algebra $C^*(E)$. There is a unique
*-homomorphism  $L (E,C) \rightarrow
C^*(E,C)$ sending the generators of $L(E,C)$ to their canonical
images in $C^*(E,C)$. This map is injective by \cite[Theorem
3.8(1)]{ag12}.

A remarkable difference between $L(E,C)$ and $C^*(E,C)$ is that we {\it know} exactly what is the
structure of the monoid $\V (L_K(E,C))$ for any separated graph
$(E,C)$, but we still do not know the structure of the monoid
$\V (C^*(E,C))$, although it is conjectured in \cite{ag11} that the
natural map $L (E,C)\to C^*(E,C)$ induces an isomorphism
$\V (L(E,C))\to \V (C^*(E,C))$, see \cite[Problem 7.6]{ag11}.

We will need the definition of $M(E,C)$ only for finitely separated
graphs. The reader can consult \cite{ag12} for the definition in the
general case. Let $(E,C)$ be a finitely separated graph, and let
$M(E,C)$ be the abelian monoid given by generators $a_v$, $v\in
E^0$, and relations $a_v=\sum _{e\in X} a_{s(e)}$, for $X\in C_v$,
$v\in E^0$. Then there is a canonical monoid homomorphism $M(E,C)\to
\V (L(E,C))$, which is shown to be an isomorphism in
\cite[Theorem 4.3]{ag12}. 

The graph $C^*$-algebra $C^*(E,C)$, as well as the algebra $L(E,C)$, have a natural structure of an amalgamated free product. In particular no groupoid model is known for them. We may produce a class of groupoid algebras associated to a separated graph by considering the following construction from \cite{ae14}.

\begin{definition}
	\label{def:algebrasOEC}
	Let $(E,C)$ be a separated graph. We then define 
	$$\mathcal O (E,C) = C^*(E,C)/J$$
	where $J$ is the closed two-sided ideal of $C^*(E,C)$ generated by all the commutators $[uu^*,vv^*]$, where $u,v$ range on all elements of the multiplicative subsemigroup $U$ of $C^*(E,C)$ generated by $E^1\cup (E^1)^*$. 
	Similarly we define
	$$\Lab (E,C) = L(E,C)/J_{\mathrm{alg}},$$
	where $J_{\mathrm{alg}}$ is the two-sided ideal of $L(E,C)$ generated by all the commutators  $[uu^*,vv^*]$ as above.  	
\end{definition}

The above construction corrects a ``defect" which the algebras $L(E,C)$ and $C^*(E,C)$ might have. Namely it is not true in general that the product of two partial isometries of a $C^*$-algebra is again a partial isometry. The exact condition needed for the product $uv$ of two partial isometries $u,v$ of a $C^*$-algebra to be a partial isometry is that the projections $vv^*$ and $u^*u$ commute.

Indeed if $vv^*$ and $u^*u$ commute then
$$(uv)(uv)^* (uv) = u(vv^*)(u^*u)v = (uu^*u)(vv^*v) = uv,$$
so $uv$ is a partial isometry. This holds in any $*$-ring, and in fact in any $*$-semigroup.
The reverse implication uses the fact that $A$ is a $C^*$-algebra, see \cite[Proposition 12.8]{ExelBook}.

Returning to our graph $C^*$-algebras $C^*(E,C)$, whenever $v\in E^0$ is such that $C_v$ contains two different elements $X$ and $Y$ such that $|X|\ge 2$ and $|Y|\ge 2$, and $e\in X$, $f\in Y$, we have that 
$e^*$ and $f$ are partial isometries but $e^*f$ is not a partial isometry, because $ee^*$ and $ff^*$ do not commute. With our construction of $\mathcal O (E,C)$, we get that the $*$-subsemigroup $S$ of $\mathcal O (E,C)$ generated by $E^0\cup E^1$ is an {\it inverse semigroup}, which means that all elements of $S$ are partial isometries and that all idempotents in $S$ mutually commute. This makes the algebras $\mathcal O (E,C)$ amenable to be studied using the theory of étale groupoids, following the model introduced by Exel in \cite{Exel08}. We refer the reader to \cite{law} for the general theory of inverse semigroups.

\subsection{Examples}

First we note that all separated graph $C^*$-algebras can be reduced, modulo Morita equivalence, to the study of bipartite graphs. A {\it bipartite graph} is a directed graph $E$ such that $E^0= E^{0,0} \sqcup E^{0,1}$ and all arrows in $E$ start in $E^{0,1}$ and end in $E^{0,0}$, that is, for all $e\in E^1$ we have 
$s(e)\in E^{0,1}$ and $r(e)\in E^{0,0}$. 

\begin{proposition} {\rm \cite[Proposition 9.1]{ae14}}
	\label{prop:reductiontobipartitegraphs} Let $(E,C)$ be a separated
	graph. Then there exists a bipartite separated graph
	$(\widetilde{E},\widetilde{C})$ such that
	$$L(\widetilde{E},\widetilde{C}) \cong  M_2(L(E,C)),\qquad
	C^*(\widetilde{E},\widetilde{C}) \cong M_2(C^*(E,C)).$$  Moreover we
	have
	$$\Lab (\widetilde{E}, \widetilde{C}) \cong M_2(\Lab
	(E,C)),\qquad \mathcal O (\widetilde{E}, \widetilde{C})\cong
	M_2(\mathcal O (E,C)) .$$
\end{proposition}

{\it Sketch of proof.}
Let $V_0$ and $V_1$ be two disjoint copies of $E^0$, and denote the
canonical maps $E^0\to V_i$ by $v\mapsto v_i$ for $i=0,1$.  Write
$\widetilde{E}^{0,0}= V_0$ and $\widetilde{E}^{0,1}=  V_1$. Now
$\widetilde{E}^1$ will be the disjoint union of a copy of $E^0$ and
a copy of $E^1$: $$\widetilde{E}^1= \{h_v\mid v\in E^0 \} \bigsqcup
\{e_0\mid e\in E^1 \},$$ with
$$
\tilde{r}(h_v)=v_0, \quad \tilde{s}(h_v)=v_1, \quad \tilde{r}(e_0)=
r(e)_0, \quad \tilde{s}(e_0)= s(e)_1, \qquad (v\in E^0, e\in E^1 )
.$$ 
For $v\in E^0$, define $C_{v_0}=\{\{e_0 \mid e\in r^{-1}(v)\} , \{h_v\} \}$, and $C_{v_1} = \emptyset$.

Denote by $e_{ij}$, $0\le i,j\le 1$ the standard matrix units of
$M_2(K)$.  We define maps $\varphi \colon L(\widetilde{E},
\widetilde{C})\to M_2(L(E,C))$ and $\psi \colon M_2(L(E,C))\to
L(\widetilde{E},\widetilde{C})$ by the rules
$$\varphi (v_i)= v\otimes e_{ii},\quad \varphi (h_v)= v\otimes
e_{01},\quad \varphi (e_0)= e\otimes e_{01}, \qquad (v\in E^0, e\in
E^1, i=0,1) ,$$ and
$$\psi (v\otimes e_{ii}) =v_i,\qquad \psi (v\otimes e_{01})=h_v,
\qquad (v\in E^0, i=0,1) ,$$
\begin{align*}\psi (e\otimes e_{00}) & =
	e_0h_{s(e)}^*,\qquad \, \psi (e\otimes e_{11})  = h ^*_{r(e)}e_0\\
	\psi
	(e\otimes e_{01}) & =e_0,\qquad \qquad \psi (e\otimes e_{10})  =
	h^*_{r(e)} e_0 h^*_{s(e)}, \qquad (e\in E^1) .
\end{align*}
Then $\varphi$ and $\psi$ give well-defined $*$-homomorphisms which are
mutually inverse. Hence we obtain $L(\widetilde{E},\widetilde{C}) \cong  M_2(L(E,C))$.
The other isomorphisms are shown in the same way. \qed.

In particular, since any classical graph C*-algebra $C^*(E)$
satisfies that the final projections mutually commute, that is, $[e(u),e(u')]=0$ for all
$u,u'\in U$, we see that $\mathcal O (E)= C^*(E)$ can always be
interpreted (through a very concrete Morita equivalence) as a graph
C*-algebra of a bipartite separated graph. A similar statement
applies to Leavitt path algebras.

We start our description of concrete examples with a motivational
example for the entire theory of separated graphs, see \cite{ag11},
\cite{ag12}, \cite{aek13}, \cite{ae14}.

\begin{example}
	\label{exam:m,ndyn-system}
	For integers $1\le m\le n$, we define the separated graph
	\newline $(E(m,n),C(m,n)) $, where
	\begin{enumerate}
		\item $E(m,n)^0 := \{v,w\}$ (with $v\ne w$).
		\item $E(m,n)^1 :=\{\alpha_1,\dots , \alpha_n,\beta_1,\dots ,\beta_m\}$ (with $n+m$ distinct edges).
		\item $s(\alpha_i)=s(\beta _j) =w$ and $r(\alpha _i)=r(\beta _j)=v$
		for all $i$, $j$.
		\item $C(m,n)= C(m,n)_v := \{X,Y\}$, where $X:= \{\alpha_1,\dots ,\alpha_n\}$ and $Y:=  \{\beta _1,\dots, \beta _m \}$.
	\end{enumerate}
	See Figure \ref{fig:m,nsepargraph} for a picture in the case $m=2$,
	$n=3$. By \cite[Proposition 2.12]{ag12},
	$$L(E(m,n),C(m,n))\cong
	M_{n+1}(L(m,n)) \cong M_{m+1}(L(m,n)),$$ where $L(m,n)$ is the
	classical Leavitt algebra of type $(m,n)$. The same argument (by way
	of universal properties) shows that
	\begin{equation*}  \label{isomatUnc}
		C^*(E(m,n),C(m,n)) \cong M_{n+1}(U^{\text{nc}}_{m,n})\cong
		M_{m+1}(U^{\text{nc}}_{m,n}) \,,
	\end{equation*}
	where $U^{\text{nc}}_{m,n}$ denotes the C*-algebra generated by the
	entries $u_{ij}$ of a universal unitary $m\times n$ matrix $U= (u_{ij})$, as studied by
	Brown and McClanahan in \cite{Brown, MCl92, MCl93}. A concrete isomorphism from the corner $wC^*(E(m,n),C(m,n))w$ to $U_{m,n}^{\mathrm{nc}}$ is given by sending $\alpha_i^*\beta_j$ to $u_{ij}$ and $\beta_j^*\alpha_i$ to $u_{ij}^*$.   
	
	The algebras $\Lab (m,n):= \Lab (E(m,n), C(m,n))$ provide natural examples
	of actions on a compact Hausdorff space
	supporting $(m,n)$-paradoxical decompositions (see Subsection \ref{subsect:type-semigroup} below).
\end{example}

\begin{figure}[htb]
	\begin{tikzpicture}[scale=4]
		\node (v) at (0,1)  {$v$};
		\node (w) at (0,0) {$w$};
		\draw[<-,red] (v.west) .. controls+(left:9mm) and +(up:.1mm) ..
		(w.north);
		\draw[<-,red] (v.west) .. controls+(left:6mm) and +(up:1mm) ..
		(w.north);
		\draw[<-,red] (v.west) .. controls+(left:3mm) and +(up:2mm) ..
		(w.north);
		\draw[<-,blue] (v.east) .. controls+(right:7mm) and +(up:1mm) ..
		(w.north);
		\draw[<-,blue] (v.east) .. controls+(right:3.5mm) and +(up:2mm) ..
		(w.north);
	\end{tikzpicture}
	\caption{The separated graph $(E(2,3),C(2,3))$}
	\label{fig:m,nsepargraph}
\end{figure}
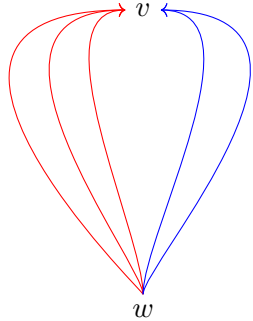

We now give an example related to the universal $C^*$-algebra generated by a partial isometry.

\begin{example}
	\label{exam:alggenepi} Let $(E,C)$ be the separated graph described
	in Figure \ref{fig:partiisometry}, with $C_v=\{ X, Y\}$ and $X=\{
	\alpha _1,\alpha_2\}$, $Y=\{ \beta _1,\beta _2 \}$. By
	\cite[Lemma 5.5(a)]{Aone-rel}, the corner $vC^*(E,C)v$ is the unital
	universal C*-algebra generated by a partial isometry. See \cite{BN}
	for a study of the {\it non-unital} universal C*-algebra
	generated by a partial isometry (of which $vC^*(E,C)v$ is the
	unitization).

	\begin{center}{
			\begin{figure}[htb]
				\begin{tikzpicture}[scale=2]
					\node (v) at (1,1)  {$v$};
					\node (w_1) at (1,0) {$w_1$};
					\node (w_2) at (0,0) {$w_2$};
					\node (w_3) at (2,0) {$w_3$};
					\draw[<-,blue]  (v.west) ..  node[above]{$\alpha_2$} controls+(left:3mm) and +(up:3mm) ..
					(w_2.north) ;
					\draw[<-,blue] (v.south) .. node[below, left]{$\alpha_1$}  controls+(left:4mm) and +(up:5mm) ..
					(w_1.north);
					\draw[<-,red] (v.south) .. node[below, right]{$\beta_1$}
					controls+(right:4mm) and +(up:5mm) ..
					(w_1.north);
					\draw[<-,red] (v.east) .. node[above]{$\beta_2$}
					controls+(right:3mm) and +(up:3mm) ..
					(w_3.north);
				\end{tikzpicture}
				\caption{The separated graph of a partial isometry}
				\label{fig:partiisometry}
		\end{figure}
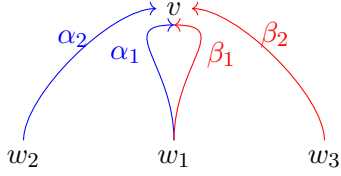}
	\end{center}

	The C*-algebra $v\mathcal O (E,C) v$ is the C*-algebra of the {\it
		free monogenic inverse monoid}, which was studied, among other
	places, in \cite{HR}. 
\end{example}

\begin{example}
	\label{exam:lamplighter} Let $(E,C)$ be the separated graph
	described in Figure \ref{fig:lampgroup}, with $C_v=\{ X, Y\}$ and
	$X=\{ \alpha _1,\alpha_2\}$, $Y=\{ \beta _1,\beta _2 \}$. By \cite[Lemma 5.5(b)]{Aone-rel}, 
	we have
	$$ vC^*(E,C)v\cong C^*( (* _{\Z}\Z_2)\rtimes \Z) \,$$
	where $\Z$ acts on $*_{\Z} \Z_2$ by shifting the factors of the
	free product. Similar computations give that $vL(E,C)v$ is
	isomorphic to the group algebra $\C [(*_{\Z}\Z_2)\rtimes \Z]$.
	
\begin{center}{
			\begin{figure}[htb]
				\begin{tikzpicture}[scale=1.5]
                \node (v) at (0,1)  {$v$};
					\node (w_1) at (-1,0) {$w_1$};
					\node (w_2) at (1,0) {$w_2$};
				        \draw[<-,blue] (v) -- (w_1) node[midway, below] {$\alpha_1$};
                        \draw[<-,blue] (v) -- (w_2) node[midway, below] {$\alpha_2$};
                       \draw[<-,bend right=30,red] (v) to node[midway,above] {$\beta_1$} (w_1);
                       \draw[<-,bend left=30,red] (v) to node[midway,above] {$\beta_2$} (w_2);
                     \end{tikzpicture}
				\caption{The separated graph underlying the lamplighter group}
				\label{fig:lampgroup}
		\end{figure}
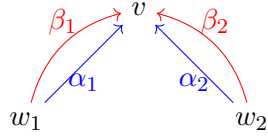}
	\end{center}

	It is easy to show that
	$$v\mathcal O (E, C)v\cong C^*(\Z_2\wr \Z), \qquad v\Lab (E,C)v\cong \C [\Z_2\wr \Z] , $$ where $\Z_2 \wr \Z$ is the lamplighter group, see Definition \ref{def:lamplighter}. The $C^*$-algebra $ C^*(\Z_2\wr \Z)$ is isomorphic to the crossed product $C^*$-algebra $C(\{0,1\}^\Z)\rtimes \Z$ associated to the two-sided shift on the Cantor space $\{0,1\}^\Z$ of biinfinite sequences on $\{0,1\}$. This follows upon identifying 
	the group $C^*$-algebra $C^*[\bigoplus_{\Z}\Z_2]$ with the $C^*$-algebra $C(\{0,1\}^\Z)$, since $\{0,1\}^\Z$ is the {\it Pontrjagin dual} of the discrete group $\bigoplus_{\Z}\Z_2$. 
\end{example}

\subsection{The groupoid model}
\label{subsect:groupoid-model}

It has been shown in complete generality in \cite{abc25} that the tame $C^*$-algebras $\mathcal O (E,C)$ admit a groupoid model. The corresponding groupoid is always Hausdorff. The structure of this groupoid is a special one, and comes from a structure of $\mathcal O (E,C)$ as a partial crossed product. In these notes, we will sketch the construction of the partial action in a particular case, namely in the case of a finite bipartite separated graph. We will always assume that a finite bipartite separated graph $(E,C)$ satisfies that $s(E^1)=E^{0,1}$and $r(E^1)= E^{0,0}$. The groupoid has better properties whenever the finite bipartite separated graph $(E,C)$ satisfies an additional rule, the ``three-twos'' rule: 
\begin{enumerate}
\item[(TT1)] $|C_v| \geq 2$ for all $v \in E^{0,0}$; 
\item[(TT2)] $|X| \geq 2$ for all $X \in C$;
\item[(TT3)] $|s^{-1}(w)| \geq 2$ for all $w \in E^{0,1}$.
\end{enumerate}

For instance, the separated graphs $E(m,n)$ for $2\le m\le n$ satisfy these conditions. 

We first recall the definition of a partial action of a group, following \cite{ExelBook}.  The definition can be given in various categories, the most basic of them is the category of sets. But we are also interested in partial actions on topological spaces, on $*$-algebras, and on $C^*$-algebras. We will give first the definition for sets, and then we will indicate the necessary upgradings for the remaining structures.  

\begin{definition} {\rm \cite[Definition 2.1]{ExelBook}}
	A partial action of a group $G$ on a set $X$ is a pair $\theta = (\{D_g\}_{g\in G}, \{\theta_g \}_{g\in G})$, where $D_g$ are subsets of $X$ and 
	$$\theta_g \colon D_{g^{-1}} \to D_g$$
	are maps such that
	\begin{enumerate}
		\item[(i)] $D_1= X$ and $\theta_1$ is the identity,
		\item[(ii)] $\theta_g\circ \theta _h\subseteq \theta _{gh}$ for all $g,h\in G$. 
	\end{enumerate}
\end{definition}

Property (ii) says that the map $\theta_{gh}$ extends the composition $\theta_g\circ \theta_h$ of the partially defined maps $\theta_g$, $\theta_h$. Recall that for two partially defined maps $f,g$ on $X$, the domain of $f\circ g$ is defined as $g^{-1}(\text{dom} (f))= \{x\in \text{dom}(g) : g(x)\in \text{dom} (f)\}$, and the range of $f\circ g$ is defined as $\text{range} (f\circ g) = f(\text{range} (g)\cap \text{dom}(f))$. It follows that each map $\theta_g$ is bijective, and $\theta_g^{-1} = \theta_{g^{-1}}$
\cite[Proposition 2.4]{ExelBook}.

If $X$ is a topological space, then a topological partial action of $G$ on $X$ is a partial action $\theta$ such that all sets $D_g$ are open subsets of $X$ and all maps $\theta_g$ are homeomorphisms. 
If $A$ is a $*$-algebra, then we require that all the sets $D_g$ are $*$-ideals of $A$, and that all maps $\theta_g$ are $*$-algebra isomorphisms. Finally, if $A$ is a $C^*$-algebra, we require all the subsets $D_g$ to be closed two-sided ideals of $A$, and all maps $\theta_g$ to be $*$-isomorphisms. 


We can generalize Example \ref{exam:transformation-groupoid} to get also a transformation groupoid from a topological partial action.

\begin{example}
	\label{exam:partialtransformation-groupoid}
	Let $\theta = (\{D_g\}_{g\in G}, \{\theta_g\}_{g\in G})$ be a topological partial action of a discrete group $G$ on a compact set $X$. Then the transformation groupoid associated to the partial action $\theta$ is the groupoid 
	$\mathcal G_{G,X}$ whose underlying set is  
	$$\{ (g,x)\in G\times X : x\in D_{g^{-1}}\},$$ 
	with unit space $\{e\}\times X\cong X$, and with  $s(g,x) = x$, $r(g,x)= gx$, and
	$$(h,y) (g,x) = (hg,x)$$
	whenever $y=gx$.  
\end{example}

Again, the groupoid $C^*$-algebra associated to the transformation groupoid of a partial action of $G$ on $X$ is isomorphic with the {\it partial crossed product} $C^*$-algebra, denoted by $C(X)\rtimes_\theta G$. This is defined in the same way as the ordinary crossed product, but in the construction of 
the algebraic partial crossed product  $*$-algebra $C(X)\alrtimes_\theta G$, one only takes elements of the form $f\delta_g$, where $f\in C_0(D_g)$ and $g \in G$, see \cite[Chapters 8 and 11]{ExelBook}. 

Theorem \ref{thm:cross-product} can be generalized to the case of partial actions, as follows:

\begin{theorem}{\rm \cite{Abadie}}
	\label{thm:partial-cross-product}
	Let $\theta$ be a topological partial action of a discrete group $G$ on a compact set $X$. Then 
	$$C^*(\mathcal G_{G,X} ) \cong C(X)\rtimes_\theta G.$$
\end{theorem}

Let $(E,C)$ be a finite bipartite separated graph. We now define the space $\Omega(E,C)$ following \cite[Section 8]{ae14}.

\begin{definition}\label{definition:local_configurations}
	Let $\mathbb{F}$ be the free group on the set $E^1$. Given $P \in 2^{\mathbb{F}}$ and $\alpha \in P$, the \textit{local configuration} $P_{\alpha}$ of $P$ at $\alpha$ is defined as the set
	$$P_{\alpha} := \{ \sigma \in E^1 \cup (E^1)^{-1} \mid \alpha \cdot \sigma \in P\}.$$
\end{definition}

\begin{definition}\label{definition:space_configurations}
	The space $\Omega(E,C)$ is defined as the space of all \textit{configurations} $\xi \in 2^{\mathbb{F}}$, which satisfy the following properties:
	\begin{enumerate}[(a)]
		\item $1 \in \xi$;
		\item $\xi$ is \textit{left convex}, i.e., if $x_1 \cdots x_n \in \xi$, where $x_1 \cdots x_n$ is a reduced word in $\mathbb{F}$ (so that $x_i \in E^1 \cup (E^1)^{-1}$ and $x_i \neq x_{i+1}^{-1}$ for $1 \leq i < n$), then $x_1 \cdots x_m \in \xi$ for all $1 \leq m \leq n$;
		\item for each $\alpha \in \xi$, the local configuration $\xi_{\alpha}$ is of exactly  one of the following forms:
		\begin{enumerate}
			\item[(c.1)] $\xi_{\alpha} = \{e^{-1} \mid e \in s^{-1}(w)\}$ for some source vertex $w \in E^{0,1}$;
			\item[(c.2)] $\xi_{\alpha} = \{e_X \mid X \in C_v\}$ for a single choice of an edge $e_X \in X$ for each $X \in C_v$, for some range vertex $v \in E^{0,0}$.
		\end{enumerate}
	\end{enumerate}
\end{definition}
In words, $\Omega(E,C)$ is the space of all subsets $\xi\subset\mathbb{F}$ containing $1$, being left convex, and whose local configurations at each point $\alpha$ consist of either all inverse edges having source a single vertex $w \in E^{0,1}$, or a choice of exactly one edge $e_X \in X$ for each $X \in C_v$, $v \in E^{0,0}$.

With the induced topology of $2^{\mathbb{F}}$, the space $\Omega(E,C)$ is a totally disconnected compact metrizable space. Thus, a basis of compact open sets for the topology is given by the ``cylinder sets''
$$\Omega_T := \{ \xi \in \Omega(E,C) \mid T \subseteq \xi\} \quad \text{for a finite subset } T\subseteq \mathbb{F} .$$
We denote the set of cylinders of the configuration space by $\mathcal Z (E,C)$ and if $T = \{ \alpha\}$, $\alpha \in \mathbb{F}$, we will simply write $\Omega_{\alpha}$ instead of $\Omega_{\{\alpha\}}$.


For $v \in E^0$, the space $\Omega_v$ is defined to be the space of all configurations ``based'' at $v$, that is, the set of all configurations $\xi \in \Omega(E,C)$
such that $\xi_1$ is of the form (c.1) for the specific source vertex $v$ in case $v \in E^{0,1}$ and of the form (c.2) for the specific range vertex $v$ in case $v \in E^{0,0}$.

We are now going to define a partial action of the free group $\F$ on $\Omega (E,C)$. 
For each $\alpha \in \F$, we have the open compact cylinder set $\Omega_{\alpha}$, and we have the homeomorphism
$$\theta_{\alpha} \colon \Omega_{\alpha^{-1}} \to \Omega_{\alpha}, \quad \theta_{\alpha}(\xi) = \alpha \cdot \xi,$$
where $\alpha \cdot \xi = \{\alpha \cdot \beta \mid \beta \in \xi\}$. For $\alpha = 1$, $\Omega_1 = \Omega(E,C)$ and $\theta_1 = \mathrm{id}_{\Omega(E,C)}$.

Geometrically, $\theta_\alpha (\xi)$ consists of a translation, that is, a change of the base point in the configuration $\xi$.

A word about notation. For two elements $\alpha,\beta\in \F$, we denote by $\alpha\cdot \beta$ the product of $\alpha$ and $\beta$ done in the free group $\F$ (which possibly includes cancellation of terms), and we just write $\alpha \beta$ to indicate the concatenation of $\alpha$ and $\beta$. Hence $\alpha\cdot \beta =\alpha \beta$ if and only if the last letter of $\alpha$ is not the inverse of the first letter of $\beta$. 

  The conditions (TT1) and (TT3) together imply that each configuration $\xi \in \Omega (E,C)$ is infinite in all the directions, that is, for each element $\alpha \in \xi$ there exists another element 
$\beta = \alpha x \in \xi$, where $x\in E^1\cup (E^1)^{-1}$. The condition (TT2) guarantees that the isotropy of the associated groupoid is topologically small (see below). The three conditions together guarantee that the space $\Omega (E,C)$ is a Cantor set. 

\begin{theorem} {\rm{\cite{ae14}}}
	\label{thm:charac-tame}
	Let $(E,C)$ be a finite bipartite separated graph. Then we have a $*$-isomorphism
	$$\mathcal O (E,C) \cong C(\Omega (E,C)) \rtimes _\theta \F \cong C^* (\mathcal G_{\F,\Omega(E,C)})$$
	for the partial action $\theta $ of $\F$ on $\Omega(E,C)$ defined above. 
	
	The isomorphism sends $p_v$ to $1_v \delta_1$ for $v\in E^0$, and $s_e$ to $1_{\Omega_e}\delta_e$ for $e\in E^1$. 
\end{theorem}

It follows that a concrete groupoid model for $\mathcal O (E,C)$ can be described as follows. 

Let $(E,C)$ be a finite bipartite separated graph satisfying the three-twos rule. Let $\mathcal G_{E,C}$ be the groupoid consisting of all the triples $(\xi',g,\xi)$, where $\xi,\xi'\in \Omega (E,C)$, $g^{-1}\in \xi$ and $g\cdot \xi = \xi'$.  The product is defined by
$$(\xi'',h,\xi')(\xi',g,\xi) = (\xi'', h\cdot g, \xi),$$
and the inverse of $(\xi', g,\xi)$ is $(\xi, g^{-1},\xi')$. 
The space $\mathcal G_{E,C}^0$ of units is $\Omega (E,C)$, and $s(g\cdot \xi,g,\xi)= \xi$, $r(g\cdot \xi,g,\xi)= g\cdot \xi$. 

The topology of $\mathcal G_{E,C}$ has a basis of open compact sets given by the sets $\Omega_{(T',g,T)}$, where $T$ is a finite, left convex subset of $\F$, $g^{-1}\in T$ and $T'=g\cdot T$, and where 
$$\Omega_{(T',g,T)} = \{ (g\cdot \xi,g,\xi): \xi \in \Omega_T\}. $$
Of course, the groupoid $\mathcal G _{E,C}$ coincides with the transformation groupoid $\mathcal G_{\F,\Omega(E,C)}$ associated to the partial action of $\F$ on $\Omega (E,C)$. 

Concerning the description of the tame Leavitt path algebra $\Lab (E,C)$, one can use the same groupoid $\mathcal G_{E,C}$ as above. In this case one recovers the $*$-algebra $\Lab (E,C)$ as the {\it Steinberg algebra} $A(\mathcal G _{E,C})$, which is the algebra of locally constant, compactly supported functions on $\mathcal G_{E,C}$, endowed with the convolution product.  We have

\begin{theorem} {\rm{\cite{ae14}}}
	\label{thm:charac-tameLeavitt}
	Let $(E,C)$ be a finite bipartite separated graph. Then we have  $*$-isomorphisms
	$$\Lab  (E,C) \cong A(\Omega (E,C)) \alrtimes _\theta \F \cong A(\mathcal G_{E,C}).$$
	for the partial action $\theta $ of $\F$ on $\Omega(E,C)$ and the groupoid $\mathcal G_{E,C}$ defined above. 
\end{theorem}

\subsection{The type semigroup of a separated graph I}
\label{subsect:type-semigroup}

We are now interested in the properties of the type semigroup associated to the groupoid model of a separated graph. This semigroup can be computed in the case of a finite bipartite separated graph, using a construction from \cite{ae14}, see Section \ref{sect:specialprops} for details. When applied to the $(m,n)$-dynamical system (Example \ref{exam:m,ndyn-system}) this produces examples of $(m,n)$-paradoxes. 

Before we give the promised description, we need to introduce the subject of $(m,n)$-paradoxes. 

Let $G$ be a group acting on a set $X$. In the setting of the classical Banach-Tarski paradox, one is interested in the study of equidecomposability for subsets of $X$. Recall that two subsets $U$ and $V$ of $X$ are {\it $G$-equidecomposable} if there are decompositions
$$U=U_1\sqcup U_2\sqcup \cdots \sqcup U_n, \qquad V=V_1\sqcup V_2\sqcup \cdots \sqcup V_n$$
and elements $g_1,\dots , g_n\in G$ such that $g_iU_i=V_i$ for $i=1,\dots , n$. 

A subset $E$ of $X$ is called $G$-paradoxical if there are two disjoint subsets $E_1$ and $E_2$ of $E$ such that each $E_i$ is $G$-equidecomposable to $E$. Tarski's Theorem states that a subset $E$ of $X$ is {\it not} $G$-paradoxical if and only if there is a finitely additive invariant measure $\mu\colon \mathcal P (X)\to [0,\infty]$ such that $\mu (E)=1$.  Tarski's Theorem can be
expressed as a result on the type semigroup $\Typ (X,G)$ associated to the action of $G$ on $X$, see
\cite[Theorem 11.1 and Corollary 11.2]{TomWagon}. This type semigroup coincides with the type semigroup associated to the transformation groupoid of the action which, in this case, is a discrete groupoid.

As Tarski shows in
\cite[Theorem 16.12]{Tarski}, his result holds also if we replace
global actions with partial actions. 

In the topological setting, that is, when $G$ is a discrete group acting by (partial) homeomorphisms on a locally compact Hausdorff topological space $X$, we may restrict the sets considered to the family $\mathbb K$ of open compact subsets of $X$. In this way, a version of the type monoid may be constructed, see for instance \cite[Section 7]{ae14}. This type monoid will be denoted here by
$\Typ (X,G,\mathbb K)$ in order to distinguish it from the set-theoretical type monoid $\Typ (X,G)$. Note that $\Typ (X,G,\mathbb K)\cong \Typ (\mathcal G _{G,X})$, where $\mathcal G_{G,X}$ is the transformation groupoid of the (partial) action.

The proof of Tarski's Theorem is based on
two special properties of the type semigroup $\Typ (X,G)$, and a general semigroup-theoretic result.

If $S$ is a commutative monoid, we consider the so-called {\it algebraic pre-order} on $S$, defined by $x\le y$ if and only if there is $z\in S$ such that $y=x+z$. In the present notes, we only consider the algebraic pre-order, but of course there are in general other pre-orders on semigroups, which may be of interest in specific contexts.  

The crucial observation linking paradoxicality of subsets of $X$ with the type monoid $\Typ (X,G)$ is that $E\subseteq X$ is $G$-paradoxical if and only if $2[E]\le [E]$ in the type semigroup $\Typ (X,G)$.  
Observe also that finitely additive invariant measures $\mu$ correspond exactly with semigroup homomorphisms $\Typ (X,G)\to [0,\infty]$.

With these preliminaries, we can state Tarski's dichotomy as follows:

\begin{theorem}
	\label{thm:TarskiDichotomy}
	Let $G$ be a group acting on a set $X$. For a subset $E$ of $X$, the following conditions are equivalent:
	\begin{enumerate}
		\item[(a)] $[E]$ is not paradoxical, that is, $2[E]\nleq [E]$ in $\Typ (X,G)$.
		\item[(b)] There exists a monoid homomorphism $\mu\colon \Typ (X,G)\to [0,\infty]$ such that $\mu ([E])= 1$.   
	\end{enumerate}
\end{theorem}

On the other hand, Tarski's semigroup theoretic result reads as follows:

\begin{theorem} {\rm{\cite[Theorem 11.1]{TomWagon}}}
	\label{thm:tarski-semigroup}
	Let $S$ be a commutative monoid and let $e\in S$.
	Then the following conditions are equivalent:
	\begin{enumerate}
		\item[(a)] For any $n\in \N$, $(n+1)e \nleq ne$.
		\item[(b)] There is a monoid homomorphism $\mu \colon S \to [0,\infty]$ such that $\mu (e)= 1$.  
	\end{enumerate}
\end{theorem}

This theorem has been generalized to arbitrary commutative pre-ordered monoids, see for instance \cite[Corollary 2.18]{KMP25}. It is worth to mention here that this general result is a consequence of the 
injectivity of the monoid $[0,\infty]$ in the category of positively ordered commutative monoids, see \cite[Example 3.10]{weh1992}. 

We now see that Tarski's dichotomy for an arbitrary commutative monoid $S$ holds exactly when $S$ has {\it plain paradoxes}, a term borrowed from \cite{KMP25}, as follows.  

\begin{definition}
	\label{def:plain-paradox} 
	Let $S$ be a commutative monoid. For $x,y\in S$, write $x<_s y$ in case there exists $n\in \N$ such that $(n+1)x\le ny$.  We say that $S$ has {\it plain paradoxes} if $x<_s x$ implies $2x\le x$.  
\end{definition}

When $S$ has plain paradoxes, for any element $x$ of $S$ such that $2x\nleq x$ we can find a monoid homomorphism $\mu \colon S\to [0,\infty]$ such that $\mu (x)=1$. In the context of operator algebras, this 
might lead to dichotomy results, such as the ones obtained in e.g. \cite{ACAHMTolich}, \cite{bl20}, \cite{KMP25}, \cite{ma} and \cite{RS}. 
Because of its connections with graph theory, we state here a result from \cite{KMP25}. We refer to that paper for all the undefined terms appearing in the statement.    

\begin{theorem} {\rm{cf. \cite[Theorem 7.15]{KMP25}}}
	\label{thm:DichotomyEP}
	Let $(G,E)$ be a self-similar action, with $E$ a row-finite graph without sources. 
	Assume that $\mathcal O _{G,E}$ is simple. Then the following conditions are equivalent:
	\begin{enumerate}
		\item $\mathcal O _{G,E}$ is not purely infinite,
		\item $\mathcal O _{G,E}$ is stably finite,
		\item $\mathcal O _{G,E}$ has a faithful semifinite lower semicontinuous trace,
		\item The graph $C^*$-algebra $C^*(E_G)$ is not purely infinite,
		\item The graph $C^*$-algebra $C^*(E_G)$ is stably finite,
		\item The type semigroup $\Typ (\mathcal G_{E,C})\cong M(E)_G = M(E_G)$ is stably finite, that is, $x+y=x$ implies $y=0$, for all $x,y\in \Typ (\mathcal G _{G,E})$.   
	\end{enumerate}
\end{theorem}

Hence we see that, although we may have that $\Typ (\mathcal G_{G,E})\ncong \V (\mathcal O _{G,E})$, the type semigroup controls important structural aspects of this large class of $C^*$-algebras. An important point here is that the graph monoid $M(E_G)$ has always plain paradoxes, indeed much better properties hold for this monoid even in the non-simple situation, as shown in \cite[Theorem 6.3, Proposition 6.4]{amp07}. We will review later these connections, in Subsection \ref{subsect:levels-of-cancellation}.   

Theorem \ref{thm:DichotomyEP} expresses an important dichotomy for the $C^*$-algebras $\mathcal O _{G,E}$ in the simple case, namely the dichotomy stably-finite/purely-infinite. Let us point out that, for simple conical commutative monoids, the validity of this dichotomy is indeed equivalent to the monoid having plain paradoxes. See \cite[Section 2]{KMP25} for a detailed account of this sort of dichotomies for preordered commutative monoids.

\begin{remark}
\label{rem:simple-dichotomy}
Suppose that $M$ is a simple conical commutative monoid. Then the following conditions are equivalent:
\begin{enumerate}
\item[(a)] $M$ has plain paradoxes.
\item[(b)] $M$ is either stably finite (i.e. $a+b= a \implies b=0$ for $a,b\in M$) or purely infinite (i.e. $2a\le a$ for every $a\in M$). 
\end{enumerate}
Indeed, suppose that (a) holds and that $M$ is not stably finite. Then there exists $a,b\in M$ such that $a+b=a$ and $b\ne 0$.
Take any nonzero $c$ in $M$. Then, by the simplicity of $M$, there exists $n,m\ge 1$ such that $a\le nc$ and $c\le mb$. Take $x\in M$ such that $nc= x+a$
We have
$$(n+1) c = nc+c \le nc + mb = x+a+mb = x+a= nc.$$
Since $M$ has plain paradoxes, we get $2c\le c$. Hence $M$ is purely infinite. 

Conversely, suppose that (b) holds, and that $(n+1) a \le na$ for some $n\ge 1$ and a nonzero element $a$ in $M$. Then $M$ is not stably finite, hence by (b) it must be purely infinite.
Hence, $2a\le a$, and thus $M$ has plain paradoxes. 
\end{remark}

In favorable circumstances the dichotomy stably-finite/purely-infinite can be improved to a dichotomy cancellative/purely-infinite in the simple setting, 
see Subsection \ref{subsect:levels-of-cancellation}.

The proof of Tarski's Theorem is based on the following properties
of $\Typ (X,G)$ (see \cite[Theorem 3.6 and Theorem 10.19]{TomWagon}):

\smallskip

(a) {\it Schr\" oder-Bernstein}: Let $x,y\in \Typ (X,G)$. If $x\le y$
and $y\le x$ then $x=y$.

\smallskip

(b) {\it $n$-cancellation}: Let $n\ge 1$ and $x,y\in \Typ (X,G)$. If
$nx=ny$ then $x=y$.

\smallskip

Any semigroup satisfying these two properties has plain paradoxes, and hence the 
semigroup obeys Tarski's dichotomy. Indeed if (a) and (b) above hold and $(n+1)e\le ne$
for some $n\ge 1$, then we have $2ne\le ne\le 2ne$ and thus $2ne= ne$ by the Schr\"oder-Bernstein property, and then since $n(2e) = ne$ we get from $n$-cancellation that $2e=e$.   

As observed in \cite[Theorem
5.4]{RorSie}, $S$ also has plain paradoxes if it satisfies the following
condition, which is quite popular amongst $C^*$-algebraists:

\smallskip

(c) {\it almost unperforation}: $(n+1)x\le ny \implies  x\le y$ for
every $n\ge 1$ and $x,y\in S$.

We now specify a possible source of non-plain paradoxes.
Given a cyclic monoid $C$,  it is easy to show that either $C$ is isomorphic to $\N$, or
there are uniquely determined
integers $m,n$ with $1\le m< n$ such that $C= \langle a \mid ma=na \rangle$. 
In the latter case, the different elements of $C$ are $0,a,2a,\dots , (n-1)a$.
Moreover, $ka=la\implies k=l$ whenever $0\le k <m$, and $ka=la\implies k=l+i(n-m)$ for some $i\in \Z$ whenever $k,l\ge m$.  In this case, we say that $C$ has {\it Leavitt type} $(m,n)$. If $C$ is cyclic infinite, we say that $C$ has Leavitt type $(1,1)$.

We say that an element $a$ of a commutative monoid $S$ is {\it $(m,n)$-paradoxical} in case the cyclic subsemigroup of $S$ generated by $a$ is of Leavitt type $(m,n)$. 
We say that the element $a$ is {\it directly finite} if $a+b= a$ in $S$ implies that $b=0$. Note that $S$ is stably finite if and only if all its elements are directly finite. 
If $S$ contains a directly finite element $a$ such that $a$ is $(m,n)$-paradoxical for some $1<m<n$, then $S$ does not have plain paradoxes, because $(m+1) a\le ma$ but $2a\nleq a$, since $a$ is directly finite.  

We now state a result, obtained in \cite{ae14}, concerning the natural homomorphism $\mathfrak{t}$ associated to a separated graph. 

\begin{theorem}
	\label{thm:typesemigroupforOmegaEC} Let $(E,C)$ be a finite
	bipartite separated graph such that $s(E^1)=E^{0,1}$ and
	$r(E^1)=E^{0,0}$, and let $(\Omega (E,C),\theta)$ be the associated partial action
	of the free group $\F$ on $E^1$ on the compact space $\Omega (E,C)$. Let $\mathcal G_{E,C}$ be the associated transformation groupoid.  
	Then the natural map
	$$\mathfrak{t}_{\mathcal G_{E,C}} \colon \Typ (\mathcal G_{E,C}) \longrightarrow \V (A (\mathcal G_{E,C}))$$ is an isomorphism. 
\end{theorem}

Here $A(\mathcal G_{E,C} )$ is the Steinberg algebra of $\mathcal G_{E,C}$, as defined above. 
It is not known whether the natural map 
$$\V (A(\mathcal G_{E,C} ) )\longrightarrow \V (C^*(\mathcal G_{E,C} ) )$$
is an isomorphism for a finite bipartite separated graph $(E,C)$. By the results in \cite{ae14}, this would be the case if the natural 
homomorphism $\V (L(E,C)) \longrightarrow \V (C^*(E,C))$ is an isomorphism for every finite bipartite separated graph $(E,C)$.
Indeed, it is shown in \cite[Section 5]{ae14} that 
\begin{equation}
\label{eq:AG-is-a direct-limit}
A(\mathcal G_{E,C})\cong \varinjlim L(E_n,C^n),\qquad C^*(\mathcal G_{E,C}) \cong \varinjlim C^*(E_n,C^n)
\end{equation}
for a certain canonical sequence $\{(E_n,C^n)\}$ of finite bipartite separated graphs. Hence the isomorphism
$\V(A(\mathcal G _{E,C})) \cong \V (C^* (\mathcal G _{E,C}))$ would follow if we knew that the natural maps 
$\V (L(E_n,C^n)) \longrightarrow \V (C^*(E_n,C^n))$ are all isomorphisms. 

See Definition \ref{def:F.infty.and.others} for the construction of the canonical sequence $(E_n,C^n)$ of finite bipartite separated graphs
associated to $(E,C)$. 

Using the first part of \eqref{eq:AG-is-a direct-limit} and \cite[Theorem 4.3]{ag12}, we obtain a monoid isomorphism 
$$\V(A(\mathcal G_{E,C})) \cong \varinjlim \V (L(E_n,C^n))   \cong \varinjlim M(E_n,C^n),$$ 
where $M(E_n,C^n)$ are the graph monoids of the separated graphs $(E_n,C^n)$.

Using these tools, the following result was obtained in \cite{ae14}. Recall that an {\it order-embedding}
of a monoid $M$ into a monoid $N$ is an injective monoid homomorphism $\iota\colon M\to N$ such that $x\le y$ if and only if $\iota (x)\le \iota (y)$ for all $x,y\in M$. 

\begin{theorem} {\rm{\cite[Theorem 7.5]{ae14}}} 
	\label{thm:main-typesemigroup} Let $M$ be a finitely generated
 conical commutative monoid. Then there exists a finite bipartite separated graph $(E,C)$
	and an order-embedding $\iota\colon M\to \Typ (\mathcal G_{E,C})$ of $M$ into the type semigroup $\Typ (\mathcal G_{E,C})$. 
\end{theorem}

\begin{remark}
\label{rem:unitary-embedding} In fact, a condition stronger than being an order-embedding, called a {\it unitary embedding}, is satisfied by the embedding $\iota 
$ of Theorem \ref{thm:main-typesemigroup}. 
\end{remark}

Take positive integers $m,n$ such that $1<m<n$. Then the construction used in the proof of 
Theorem \ref{thm:main-typesemigroup} applied to
the monoid $M=\langle a \mid ma=na\rangle$ gives rise to the  
separated graph $(E,C)=(E(m,n), C(m,n))$. Hence we obtain from Theorem \ref{thm:main-typesemigroup} that the type monoid $\Typ  (\mathcal G_{E,C} )$ supports an $(m,n)$-paradoxical decomposition, indeed the open compact set $\mathcal Z (w)$ is $(m,n)$-paradoxical.  Using globalization techniques from \cite{Abadie}, one can convert this example into an example of a global action of a group on a compact space $X$ admitting non-plain paradoxes, showing that Tarski's Dichotomy Theorem does not extend to the topological setting. See \cite[Section 7]{ae14} for details. This answered questions formulated in  \cite{KN} and \cite{RorSie}. 

We may wonder what are the properties of the groupoid $\mathcal G_{E,C}$ of the separated graph $(E,C) = (E(m,n),C(m,n))$ and of its associated algebras. Using known results in the literature, we 
state below some of its properties.
 
  Recall that a groupoid $\mathcal G$ is {\it effective} if the interior of the isotropy subgroupoid $\text{Iso} (\mathcal G) = \bigsqcup_{x\in \mathcal G^{(0)}} \mathcal G_x^x$ 
 equals the group of units $\mathcal G^{(0)}$ of $\mathcal G$. 

Recall from \cite{ae14} that there exists a {\it reduced} tame $C^*$-algebra $\mathcal O ^r (E,C)$ associated to any
finite bipartite separated graph $(E,C)$. It is defined as the reduced crossed product
$$\mathcal O ^r (E,C) := C(\Omega (E,C)) \rtimes ^r \mathbb F .$$
There is a canonical surjective $*$-homomorphism $\mathcal O (E,C) \to \mathcal O^r (E,C)$.

\begin{theorem}
\label{thm:properties-ofGmn} Let $m,n$ be integers such that $1<m<n$. Then the groupoid $\mathcal G (m,n) := \mathcal G_{E(m,n),C(m,n)}$ satisfies the following properties:
\begin{enumerate}
\item $\mathcal G (m,n)$ supports an $(m,n)$-paradoxical decomposition, that is, its type monoid $\Typ (\mathcal G (m,n))$ contains an element $u$ which is $(m,n)$-paradoxical. In particular
$\Typ (\mathcal G (m,n))$ does not have plain paradoxes.
\item The Murray-von Neumann monoid $V(A(\mathcal G (m,n)))$ of the Steinberg algebra $A(\mathcal G (m,n)) = \Lab (E(m,n),C(m,n))$ does not have plain paradoxes.
\item $\mathcal G (m,n)$ is an effective Hausdorff groupoid.
 \item $\mathcal G (m,n)$ is not a minimal groupoid.
\item The Steinberg algebra $A (\mathcal G (m,n)) = \Lab (E(m,n),C(m,n))$ is a prime algebra. Likewise the reduced $C^*$-algebra  $C^*_r (\mathcal G (m,n))= \linebreak  \mathcal O^r(E(m,n),C(m,n))$ is a prime $C^*$-algebra. 
 \item The natural map $C^* (\mathcal G (m,n)) \to C^*_r(\mathcal G (m,n))$ is not injective. In particular, $\mathcal G (m,n)$ is not an amenable groupoid.
 \end{enumerate}
\end{theorem}

\begin{proof}
(1) This follows from \cite[Theorem 7.5]{ae14}.

(2) This follows from (1) and Theorem \ref{thm:typesemigroupforOmegaEC}. 

(3) Recall that, for a finite bipartite separated graph $(E,C)$, $\mathcal G_{E,C}$ is the transformation groupoid of a partial action $\theta_{(E,C)} \colon \mathbb F \act \Omega (E,C)$.
It is well-known (and easy to show) that the transformation groupoid of a partial action $\theta$ of a discrete group on a locally compact Hausdorff space is always a Hausdorff groupoid, and that it is effective if and only if the partial action $\theta$  is topologically free. 
Now by \cite[Theorem 10.5]{ae14}, the partial action $\theta_{(E,C)}$ is topologically free if and only if $(E,C)$ satisfies the so-called condition (L) introduced in \cite[Definition 10.2]{ae14}.
Now it is straightforward to show that a finite bipartite separated graph satisfying condition (TT2) (i.e. such that $|X| \ge 2$ for all $X\in C$) satisfies condition (L). 
It follows in particular that $(E(m,n),C(m,n))$ satisfies condition (L) and consequently $\mathcal G (m,n)$ is an effective Hausdorff groupoid. 

(4) We provide here an explicit example of a non-trivial closed invariant subset of $\Omega^u := \Omega  (E(m,n),C(m,n))$, following the proof of \cite[Proposition 3.9]{aek13}. We first recall from \cite{aek13} that an $(m,n)$-dynamical system consists of a pair $(X,Y)$ of compact Hausdorff spaces such that there are clopen partitions 
$$X= \bigsqcup_{i=1}^n H_i = \bigsqcup_{j= 1}^m V_j $$
together with homeomorphisms $h_i\colon Y\to H_i$, for $i=1,\dots , n$ and $v_j\colon Y\to V_j$, for $j=1,\dots ,m$. Given an $(m,n)$-dynamical system $(X,Y, \{h_i\},\{v_j\})$, there is a canonical partial action $\theta$ of the free group $\mathbb F_{n+m}$ on the topological disjoint union 
$\Omega = X \overset{\cdot}{\cup} Y$, see \cite[Proposition 3.3]{aek13}. 

For instance, take $\Omega ^u = \Omega (E(m,n), C(m,n))$, and set 
$$ X^u :=  \Omega ^u_v  = \{\xi \in \Omega^u : \xi_1 \text{ is of the form (c.2) }\}, $$
$$Y^u :=  \Omega ^u_w  = \{\xi \in \Omega^u : \xi_1 \text{ is of the form (c.1) }\} ,$$
and $h_i^u\colon Y^u\to X^u$, $v_j^u\colon Y^u\to X^u$, defined by $h_i^u (\xi) = \alpha_i \cdot \xi$,  $v_j^u(\xi) = \beta_j \cdot \xi$, for $i=1,\dots, n$ and $j=1,\dots , m$. 
Then $(X^u,Y^u, \{h^u_i\}, \{v^u_j\})$ is an $(m,n)$-dynamical system, and it is indeed the universal $(m,n)$-dynamical system, in the sense that for each other
$(m,n)$-dynamical system $(X,Y, \{h_i\},\{v_j\})$, there exists a unique equivariant continuous map $\gamma \colon \Omega :=X\overset{\cdot}{\cup} Y \to \Omega ^u$, see \cite[Theorem 3.8]{aek13}. 

We are now ready to define our explicit example of a non-trivial closed invariant subset of $\Omega^u$. Set $p= n-m+1$, and take $Y= \{1,\dots ,p \}^{\mathbb N}$, and $X= \{1,\dots ,m \} \times Y$. 
Define the maps $v_j \colon Y\to X$ and $h_i\colon Y\to X$ as follows. First we set $v_j(y) = (j,y)$ for $j=1,\dots , m$. We now set $h_i= v_i$ for $i=1,\dots , m-1$, and 
$$h_{m-1+k} (y) = (m,ky) := (m, (k,y_1,y_2,y_3,\dots )), \qquad 1\le k\le p .$$
Then $(X,Y,\{h_i\},\{v_j\})$ is an $(m,n)$-dynamical system, and hence there is a unique equivariant continuous map $\gamma \colon \Omega := X\overset{\cdot}{\cup} Y \to \Omega^u$.
Note that the partial action of $\mathbb F_{n+m}$ on $\Omega$ is clearly minimal, but not topologically free. Indeed $Y$ is a clopen subset of $\Omega$ on which $v_1^{-1}h_1$ is the identity. 
Hence the same properties hold for its continuous image $\gamma (\Omega) \subseteq \Omega^u$. Since the partial action of $\mathbb F_{n+m}$ on $\Omega^u$ is topologically free (by (3)), it follows that $\gamma (\Omega)$ is a non-trivial closed invariant subset of $\Omega^u$, showing that the transformation groupoid $\mathcal G (m,n)$ is not minimal.

(5) This is \cite[Example 9.12(1)]{al18}.

(6) This follows from \cite[Theorem 7.2]{aek13}, see also \cite[Theorem 5.1]{lolk-nuc}. 
\end{proof}

By Theorem \ref{thm:properties-ofGmn}(4), the groupoid $\mathcal G (m,n)$ is not a minimal groupoid, equivalently the partial action of $\mathbb F_{n+m}$ on $\Omega (E(m,n),C(m,n))$ is not minimal. We propose the following open problem.

\begin{openproblem}
\label{opp:simple-groupoid}
Given integers $m,n$ with $1<m<n$,
does there exist a minimal $(m,n)$-dynamical system such that the corresponding type semigroup contains an $(m,n)$-paradoxical element? 
\end{openproblem}

It is worth recording here that using globalization techniques, one can pass from a partial action to a global action in the above example:

\begin{theorem}{\rm{\cite[Corollary 7.12]{ae14}}} There exists a global action of a finitely generated free group $\mathbb F$ on a totally disconnected 
metrizable compact space $Z$ such that the type monoid $\Typ (Z,\mathbb F, \mathbb K)$ does not have plain paradoxes. 
\end{theorem}

\begin{proof}
Take $1<m<n$. By \cite[Corollary 7.9]{ae14}, the partial action $\theta$ of $\mathbb F_{n+m}$ on $\Omega:= \Omega (E(m,n),C(m,n))$ can be globalized to an action $\beta$ of $\mathbb F_{n+m}$
on a locally compact, totally disconnected, metrizable space $Y$, in such a way that $\Typ (\Omega, \mathbb F_{n+m},\mathbb K) \cong \Typ (Y, \mathbb F_{n+m},\mathbb K)$.
Taking $Z$ as the one-point compactification of $Y$, with trivial action on the point at infinity, one gets the desired compact space $Z$.
\end{proof}

Recall that a commutative monoid $M$ is said to be a {\it refinement monoid} if whenever $a+b=c+d$ in $M$, there exist $x,y,z,t$ in $M$ such that $a=x+y$ and $b=z+t$ while $c= x+z$ and $d= y+t$. 
Also, $M$ is said to be {\it conical} in case $x+y=0$ in $M$ implies that $x=y=0$. It is shown in \cite{weh} that the type semigroup of any ample groupoid is a conical refinement monoid (see Section \ref{sect:specialprops}).

In \cite{weh}, Wehrung showed that all countable conical refinement monoids appear as type semigroups, as follows:

\begin{theorem}{\rm{\cite{weh}}}
\label{thm:wehrung-realization}
Let $M$ be a countable conical refinement monoid. Then there exist a locally compact Hausdorff metrizable  zero-dimensional space $X$ and an action of a discrete group $G$ by homeomorphisms on $X$ such that $M\cong \Typ (X,G,\mathbb K)$.
\end{theorem}

This is shown in \cite[Theorem 4.8.9]{weh}, but since both the statement and its proof are written in a different language, to which people working on operator algebras may not be used, I will give here a brief guide to the main lines of the proof.

The first thing to notice is that Wehrung works in \cite{weh} with the equivalent notion of Boolean rings. By Stone duality, this is equivalent to working with locally compact, Hausdorff, zero-dimensional topological spaces, see \cite[Section 1.4]{weh}.
Hence one has to show that any countable conical refinement monoid is isomorphic to the type monoid, denoted $\text{Typ} (\text{Inv} (B,G))$, associated to the action of a discrete group $G$ on a Boolean ring $B$. The type monoid of a Boolean inverse semigroup is defined in \cite[Definition 4.1.3]{weh} (see Section \ref{sect:specialprops} below). The inverse monoid $\text{Inv}(B,G)$ associated to an action of a group $G$ on a Boolean ring $B$ is defined in \cite[Example 4.4.15]{weh}. Now \cite[Theorem 4.8.9]{weh} asserts that every countable conical refinement monoid is group-measurable. By definition (see \cite[Definition 4.8.4]{weh}), this means that there are a Boolean ring $B$ and a surjective group-induced V-measure $\mu \colon B \to M$. By \cite[Theorem 4.7.14]{weh}, this implies that there exists a monoid isomorphism $\tau \colon \Z^+<B>//G \to M$ for a suitable action of a group $G$ on $B$, where $\Z^+<B>//G$ is a certain monoid, which by \cite[Proposition 4.4.20]{weh} is isomorphic to the type semigroup $\text{Typ} (\text{Inv} (B,G))$ associated to the action of $G$ on $B$. This concludes our little guide to the proof of the theorem. We stress that all the argument depends on a crucial result of Dobbertin \cite[Theorem 3.4]{dobbertin}, which implies that all countable conical refinement monoids are $V$-measurable.  

Using an embedding result of Wehrung, one can deduce the existence of a minimal action whose type semigroup does not have plain paradoxes, as follows:

\begin{corollary}
There exists a minimal action of a countable discrete group $G$ on a locally compact Hausdorff zero-dimensional metrizable space $X$ such that the type semigroup $\Typ (G,X,\mathbb K)$ does not have plain paradoxes.
\end{corollary}

\begin{proof} Take integers $m,n$ such that $1<m<n$, and consider the conical commutative monoid $S= \langle a \mid na=ma \rangle$.
Then $S$ is a simple monoid, and obviously $S$ does not have plain paradoxes, since $a$ is an atom in $S$, and $a$ is directly finite in $S$. 
By \cite[Corollary 2.7]{weh1998}, there exists an order-embedding $\iota \colon S \to N$, where $N$ is a simple conical refinement monoid.  By a standard argument, there exists a countable simple conical refinement submonoid $M$ of $N$ containing $\iota (S)$. Note that $2\iota (a)\nleq \iota (a)$ because $\iota $ is an order-embedding. Hence $M$ does not have plain paradoxes. Now by Theorem 
\ref{thm:wehrung-realization}, there exist a locally compact Hausdorff metrizable  zero-dimensional space $X$ and an action of a discrete group $G$ by homeomorphisms on $X$ such that $M\cong \Typ (X,G,\mathbb K)$. Since $\Typ (X,G,\mathbb K)$ is simple, the action is minimal, as desired. 
\end{proof}

\section{Special properties of type semigroups}
\label{sect:specialprops}

In this section we will state some properties of type semigroups in certain favorable circumstances. In particular, we review an interesting result of Wehrung \cite[Theorem 5.3.8]{weh}. This result goes far beyond the framework of minimal actions. In the final part of the section, we will describe how to compute the type monoid associated to any finite bipartite separated graph, and we will consider the specific example of the type monoid of the full shift, which illustrates very well this computation, and establishes a link to symbolic dynamics.

\subsection{Different levels of cancellation in commutative monoids}
\label{subsect:levels-of-cancellation}

We will follow \cite{AGOP} and the recent paper \cite{agnopp} concerning notation and terminology. 

Let $M$ be a commutative monoid. We say that $M$ is {\it cancellative} 
if $x + z = y + z$ implies $x = y$ for any $x, y, z \in  M$. We say that 
$M$ is {\it separative} (or has separative cancellation) if $2x = x + y = 2y$ implies $x = y$ for
every $x, y  \in M$. We say that $M$ is {\it strongly separative} if $2x = x + y$ implies $x = y$ for any
$x, y \in  M$.

Observe that a cancellative monoid is necessarily stably finite. 

For a commutative monoid $M$, separativity is equivalent to cancellation of ``two-sided small" elements: $M$ is separative if and only if whenever we have
$x+z=y+z$ and $z\le nx$, $z\le my$ for some $n,m\ge 1$, then we have $x=y$, see \cite[Lemma 2.1]{AGOP}. 

Similarly, strong separativity is equivalent to cancellation of ``one-sided small" elements: $M$ is strongly separative if and only if whenever we have
$x+z=y+z$ and $z\le nx$ for some $n\ge 1$, then we have $x=y$, see \cite[page 126]{AGOP}. 

Separativity leads to a strong form of dichotomy in the simple case, that is, a simple conical monoid $M$ is separative if and only if it is either cancellative or purely infinite. 
If a simple conical monoid is strongly separative then it is necessarily cancellative.

We now state the definition of a Boolean inverse monoid, see for instance \cite{weh}.

Recall that two elements $s,t$ of an inverse semigroup are said to be {\it orthogonal} if $s^*t= 0= st^*$. 

\begin{definition}
	\label{def:BIS}
	A {\it Boolean inverse semigroup} is an inverse semigroup $S$ such that the set of idempotents is a (generalized) Boolean algebra and such that every pair of orthogonal elements $x,y\in S$ has a supremum, denoted by $x\oplus y$. 
\end{definition}

We denote the semilattice of idempotents of an inverse semigroup by $\mathcal E (S)$. 

Given a Boolean inverse semigroup $S$, one can define the {\it type monoid} $\text{Typ} (S)$ as the commutative monoid with generators $\text{typ} (x)$, where $x\in \mathcal E (S)$, and relations:
\begin{enumerate}
	\item $\text{typ} (0) = 0$,
	\item $\typ (x) = \typ (y)$ whenever $x,y\in \mathcal E (S)$ and there is some $s\in S$ such that $x= s^*s$ and $y= ss^*$. 
	\item $\typ (x\oplus y) = \typ (x) + \typ (y) $ whenever $x,y\in S$ are orthogonal.  
\end{enumerate} 

The type monoid is always a conical refinement monoid, see \cite[Corollary 4-1.4]{weh}. 
Let $\mathcal G$ be an ample Hausdorff groupoid. Then the type monoid $\Typ (\mathcal G)$ of $\mathcal G$ as defined in these notes agrees with the type monoid $\Typ (B(\mathcal G))$, where $B(\mathcal G)$ is the Boolean inverse semigroup of open compact bisections of $\mathcal G$.  

As mentioned above, this type semigroup $\Typ (\mathcal G)$ also agrees with other versions of the type semigroup, as defined in \cite{bl20}, \cite{PSS}, \cite{RS}. 

In the case where $G$ is a discrete group acting on a zero-dimensional, metrizable locally compact Hausdorff space $X$, the Boolean inverse monoid associated to the transformation groupoid of the action
is isomorphic to the Boolean inverse monoid $\pHomeo (X,G)$, defined as the set of partial homeomorphisms of $X$ defined on open compact subsets of $X$ and such that are locally in $G$. More precisely, we consider finite families $(U_1,\dots ,U_n)$, $(V_1,\dots ,V_n)$, and $(g_1,\dots , g_n)$ such that $U_i,V_i$ are open compact subsets of $X$ and $g_i\in G$, for $i=1,\dots ,n$, with the property that the $U_i$'s are mutually disjoint, the $V_i$'s are mutually disjoint, and $g_i(U_i)= V_i$ for all $i$. We say that $g_1,\dots ,g_n$ is the support of the corresponding partial homeomorphism $$g:= g_1\sqcup \cdots \sqcup g_n\colon \bigsqcup_{i=1}^n U_i \to \bigsqcup_{i=1}^n V_i.$$ 

Following Wehrung \cite[Definition 5.3.1]{weh}, we define a {\it fork} in a Boolean inverse semigroup $S$ as a triple $(c,g_1,g_2)$ of elements of $S$ such that $c\in \mathcal E(S)$, $g_1,g_2\in S$, $c\le d(g_i):= g_i^* g_i$ for $i=1,2$ and $g_1\langle c \rangle g_2\langle c \rangle = 0$. Here $g\langle c\rangle = gcg^{-1}$ for $c\in \mathcal E (S)$ and $g\in S$. For $n\in \Z^+$ the expression $\langle g_1,g_2\rangle ^{-n} (c)$ denotes the product of all the elements $g^{-1}\langle c \rangle$, where $g= g_{i_1}g_{i_2}\cdots g_{i_r}$ with $r\le n$ and $i_j\in \{1,2\}$ for all $j$. Note that, setting $c_n = \langle g_1,g_2\rangle ^{-n}(c)$, we have 
$$c_{n+1} = g_1^{-1}\langle c_n \rangle  g_2^{-1} \langle c_n \rangle c_n$$
for all $n$.  

A fork $(c,g_1,g_2)$ is {\it nilpotent} if there is a non-negative integer $n$ such that 
$c_n= \langle g_1,g_2\rangle ^{-n} (c)= 0$.   

The following result, due to Wehrung, is the main technical tool needed to show Theorem \ref{thm:supramenable} below. 

\begin{theorem} {\rm \cite[Theorem 5.3.4]{weh}}
	\label{thm:nilpot-fork} Let $S$ be a Boolean inverse monoid. If every fork in $S$ is nilpotent, then the type monoid $\Typ (S)$ is strongly separative. 
\end{theorem}

Recall that a group $G$ is supramenable if no nonempty subset of $G$ is paradoxical with respect to the natural left action of $G$ on itself. 

The following theorem was proven in an equivalent way by Wehrung. We present a proof (essentially his proof)
here for convenience of the reader. 

\begin{theorem} {\rm \cite[Theorem 5.3.8]{weh}}
	\label{thm:supramenable}
	Let $G$ be a supramenable group acting by homeomorphisms on a zero-dimensional, locally compact Hausdorff space $X$. Then the type semigroup $\Typ (X,G,\mathbb K)$ is strongly separative. 
\end{theorem}

\begin{proof}
	By Theorem \ref{thm:nilpot-fork}, it suffices to show that all forks in $\pHomeo (X,G)$ are nilpotent. Set $S:= \pHomeo (X,G)$. 
	
	The set of idempotents of $S$ is the set of identities $1_K$ of open compact subsets $K$ of $X$.
	Note that for $g\in S$ we have $g1_Kg^{-1} = 1_{g(K\cap \text{dom} (g))}$. Hence we may identify $\mathcal E (S)$ with the algebra of open compact subsets of $X$ and $g\langle 1_K \rangle$ with 
	$g(K\cap \text{dom} (g))$.   
	
	Let $(c,g_1,g_2)$ be a fork of $S$. Then $c$ may be identified with an open compact subset, say $C$, of $X$, such that $C$ is contained in the domain of $g_i$, for $i=1,2$, and 
	$$g_1(C) \cap g_2(C) =\emptyset.$$  
	Now consider the sets $C_n$ defined inductively by 
	$$C_0= C, \qquad C_{n+1} = g_1^{-1}(C_n\cap \text{range} (g_1))\cap g_2^{-1}(C_n\cap \text{range} (g_2)) \cap C_n.$$
	
	Suppose by way of contradiction that $(c,g_1,g_2)$ is not nilpotent. Then we have that $C_n\ne \emptyset$. Observe that then the sequence $\{C_n\}$ is a decreasing sequence of non-empty compact subsets of $X$ and hence $K:= \cap_{n=0}^{\infty} C_n \ne \emptyset$.  Observe that $K$ is a compact subset of $X$ with $K\subset \text{dom}(g_1)\cap \text{dom} (g_2)$ and
	$$g_1(K)\cap g_2(K) \subseteq g_1(C)\cap g_2 (C) =  \emptyset.$$
	Moreover $g_i(K) \subseteq K$ for $i=1,2$. 
	We now want to ``translate" these relations to $G$, in order to contradict the hypothesis that $G$ is supramenable. For this, we fix a point $t_0\in K$ and consider the map
	$\varphi \colon \mathcal P (X) \to \mathcal P (G)$  (where $\mathcal P (A)$ denotes the powerset of a set  $A$) defined by
	$$\varphi (Y) = \{ g\in G : g(t_0)\in Y \} \qquad (Y\subseteq X).$$
	It is straightforward to check that $\varphi$ preserves arbitrary unions and intersections and sends the emptyset to itself. Observe also that $\varphi (K)\ne \emptyset$ and $\varphi (gY)= g\varphi (Y)$ for each $Y\in \mathcal P (X)$. 
	
	But now, setting $Z= \varphi (K)$ and $Z_i= \varphi (g_iK)= g_i \varphi (K)= g_iZ$ for $i=1,2$, we have   $Z\ne \emptyset $, $Z_1\cup Z_2 \subseteq Z$, $Z_1\cap Z_2 = \emptyset$ and $Z\sim_G Z_i$ for $i=1,2$. Hence $Z$ is $G$-paradoxical which contradicts the hypothesis that $G$ is supramenable.
\end{proof}

The following corollary is immediate:

\begin{corollary}
\label{cor:minimalcase-cancellative} 
Let $G$ be a supramenable group acting minimally on a compact zero-dimensional space $X$. Then the type semigroup $\Typ (X,G,\mathbb K)$ is cancellative.
\end{corollary}

\begin{proof}
By Theorem \ref{thm:supramenable}, the type semigroup $\Typ (X,G,\mathbb K)$ is strongly separative. 
Since the action is minimal, it is cancellative. 
\end{proof}

It seems to be an open problem whether the type semigroup $\Typ (X,G,\mathbb K)$ of a minimal action of an amenable group on a Cantor set is cancellative, see \cite[Proposition 2.20]{Melleray}
and the comments after it. 

There are other properties of interest concerning type semigroups. A conical monoid $M$ is said to be {\it unperforated} in case, for each positive integer $n$, we have that $na\le nb$ in $M$ implies that $a\le b$ in $M$. By a result of Chen \cite[Theorem 1]{chen}, an unperforated refinement monoid is necessarily separative. 

Any cancellative and unperforated conical refinement monoid is a direct limit of finitely generated free commutative monoids (see \cite[Theorem 3.14]{ag15}). These are precisely the 
positive cones of dimension groups. 
Generalizing the above, we may consider {\it tame refinement monoids} \cite{ag15}, which, by definition, are those refinement monoids that may be written as a direct limit of
a directed system of finitely generated refinement monoids. Tame refinement monoids are always unperforated and separative \cite[Theorem 3.14]{ag15}. If a simple conical refinement monoid is
tame, then it is either cancellative (and thus the positive cone of a dimension group) or purely infinite. Hence a strong form of the dichotomy principle holds in this case.

A refinement monoid is said to be {\it wild} if it is not tame. See \cite{ag15} and also Theorem \ref{thm:type-semigroup-lamplighter} below for some examples of wild refinement monoids related to 
specific examples of separated graphs.

We can now state an important structural result for the type monoid associated to a self-similar action. 

\begin{theorem}
\label{thm:type-monoid-self-similar}
Let $\mathcal G _{G,E}$ be the groupoid associated to a self-similar graph $(G,E)$. Then the type monoid $\Typ (\mathcal G_{G,E})$ is a tame refinement monoid.
In particular $\Typ (\mathcal G_{G,E})$ is always unperforated and separative. 
\end{theorem}

\begin{proof}
Since $\Typ (\mathcal G_{G,E}) \cong M(E_G)$, where $E_G$ is the graph from Lemma \ref{lem:coinvariants}, it suffices to show the result for a graph monoid $M(E)$. This is done in \cite[Theorem 4.1]{ag15}. 
\end{proof}

It follows in particular that a strong form of the dichotomy principle holds for the type monoids of self-similar graphs in the simple case. Indeed, if $\Typ (\mathcal G_{G,E})$ is a simple monoid, then it is either the positive cone of a dimension group or purely infinite. 

\subsection{The type semigroup of a separated graph II}
\label{subsect:type-semig-sparated-graph}

Let $(E,C)$ be a finite bipartite separated graph, and let $\mathcal G _{E,C}$ be the associated ample groupoid, see Subsection \ref{subsect:groupoid-model}. 

Recall that the type monoid of a self-similar graph $(G,E)$ is isomorphic to the graph monoid $M(E_G)$, where $E_G$ is the graph first considered by Larki in \cite{Larki}.
It was shown in \cite[Proposition 4.4]{amp07} that the graph monoid $M(E)$ of any row-finite graph $E$ is a refinement monoid. However, for a separated graph $(E,C)$, the graph monoid $M(E,C)$ is not necessarily a refinement monoid, even for finite bipartite separated graphs \cite{ag12}, and thus, since the type semigroup $\Typ (\mathcal G_{E,C})$ is always a refinement monoid, it cannot happen that $M(E,C)\cong \Typ (\mathcal G_{E,C})$ in general. 

However, the type semigroup $\Typ (\mathcal G _{E,C})$ can be computed through the construction of an infinite separated graph associated to $(E,C)$, called the {\it canonical resolution} of $(E,C)$.
This infinite separated graph is a colored version of a Bratteli diagram, called a {\it separated Bratteli diagram}, which is defined as follows.

\begin{definition}\cite{al18, ac24}\label{def:separated.Bratteli}
	A \textit{separated (or colored) Bratteli diagram} is an infinite separated graph $(F,D)$ with the following properties:
	\begin{enumerate}[(a)]
		\item The vertex set $F^0$ is the union of finite, non-empty, pairwise disjoint sets $F^{0,j}$, $j \ge 0$.
		\item The edge set $F^1$ is the union of finite, non-empty, pairwise disjoint sets $F^{1,j}$, $j\ge 0$.
		\item The range and source maps satisfy $r(F^{1,j})=F^{0,j}$ and $s(F^{1,j})=F^{0,j+1}$ for all $j\ge 0$, respectively.
	\end{enumerate}
\end{definition}

Note that a Bratteli diagram is just a separated Bratteli diagram with the trivial separation.

We now recall the construction of the canonical resolution of a finite bipartite separated graph, see \cite{ae14} and \cite{al18}.
In the following, for a finite sequence $(x_1,\dots , x_n)$ and $1\le i\le n$, we will denote by $(x_1,\dots \widehat{x_i},\dots , x_n)$ the sequence
obtained by deleting $x_i$ from the original sequence, that is $(x_1,\dots \widehat{x_i},\dots , x_n) = (x_1,\dots, x_{i-1},x_{i+1},\dots , x_n)$. 

\begin{definition}\label{def:F.infty.and.others}
	Let $(E,C)$ be any finite bipartite separated graph, and write
	$$C_u = \{X_1^u,\dots,X_{k_u}^u\}$$
	for all $u \in E^{0,0}$. Then the \textit{$1$-step resolution} of $(E,C)$ is the finite bipartite separated graph denoted by $(E_1,C^1, r_1,s_1)$, and defined by
	\begin{enumerate}[(a)]
		\item $E_1^{0,0} := E^{0,1}$ and $E_1^{0,1} := \{v(x_1,\dots,x_{k_u}) \mid u \in E^{0,0},\,  x_j \in X_j^u ,\,\, 1 \leq j \leq k_u\}$.
		\item $E_1^1 := \{\alpha^{x_i}(x_1,\dots,\widehat{x_i},\dots,x_{k_u}) \mid u \in E^{0,0}, 1 \leq i \leq k_u,\,  x_j \in X_j^u ,\,\, 1 \leq j \leq k_u \}$.
		\item $$r_1(\alpha^{x_i}(x_1,\dots,\widehat{x_i},\dots,x_{k_u})) := s(x_i)$$ 
        $$ \text{and} \qquad s_1(\alpha^{x_i}(x_1,\dots,\widehat{x_i},\dots,x_{k_u})) := v(x_1,\dots,x_{k_u}).$$
		\item $C^1_v := \{X(x) \mid x \in s^{-1}(v)\}$ for $v \in E^{0,0}_1=E^{0,1}$, where for $u \in E^{0,0}$, $1 \leq i \leq k_u$, and $x_i\in X_i^u$,
		$$X(x_i) := \{\alpha^{x_i}(x_1,\dots,\widehat{x_i},\dots,x_{k_u}) \mid x_j \in X_j^u \text{ for } j \ne i\}.$$
	\end{enumerate}
	A sequence of finite bipartite separated graphs $\{(E_n,C^n)\}_{n \ge 0}$ with $(E_0,C^0) := (E,C)$ is then defined inductively by letting $(E_{n+1},C^{n+1})$ denote the $1$-step resolution of $(E_n,C^n)$. The bipartite separated graph $(E_n,C^n)$ is called the \textit{$n$-step resolution} of $(E,C)$. Finally $(F,D)$ is the infinite layer graph
	$$(F,D) := \bigcup_{n=0}^{\infty} (E_n,C^n).$$
	It is clear by construction that $(F,D)$ is a separated Bratteli diagram, called the {\it canonical resolution} of the finite bipartite graph $(E,C)$.
\end{definition}

The main point with the canonical resolution of a finite bipartite separated graph is that the vertices in the $n$-th level of the Bratteli diagram represent the cylinder sets corresponding to an $n$-ball
in the configuration space $\Omega (E,C)$. Here an $n$-ball of $\Omega (E,C)$ is the set of all elements $\gamma \in \xi$ such that $|\gamma | \le n$, for some configuration $\xi \in \Omega (E,C)$. 

The graph monoid $M(F,D)$ is a refinement monoid, and the inclusion $(E,C) \subset (F,D)$ induces an order-embedding of monoids
$\iota \colon M(E,C) \to M(F,D)$. Indeed the monoid homomorphism $\iota$ is a unitary embedding, in the sense of \cite[Definition 3.7]{ae14}, which is a stronger property than just being an order-embedding. Moreover, by \cite[Lemma 4.5]{ae14}, the inclusion $\iota \colon M(E,C) \to M(F,D)$ is a refinement of $M(E,C)$ in the sense of \cite[Definition 4.3]{ae14}. 

We can now state a main result in \cite{ae14}.

\begin{theorem} {\rm \cite[Theorem 7.4]{ae14}}
	\label{thm:typesemigroup-bipar-separated}
	Let $(E,C)$ be a finite bipartite separated graph and let $(F,D)$ be its canonical resolution. Then there is a monoid isomorphism  $$\Typ (\mathcal G _{E,C})\cong M(F,D) \cong \varinjlim M(E_n,C^n)
    \cong \V (A(\mathcal G_{E,C})).$$ 
	\end{theorem}

We now consider a concrete example, namely the separated graph $(E,C)$ giving rise to the full shift on the alphabet 
$\{0,1\}$, which has been introduced in Example \ref{exam:lamplighter}.

\begin{example}
 \label{exam:lamplighter2}
  Let $(E,C)$ be the separated graph
described in Figure \ref{fig:lampgroup2}, with $C_v=\{ X, Y\}$ and
$B^0=\{ \alpha _0,\alpha_1\}$ and $R^0=\{ \beta _0,\beta _1 \}$. We call $\alpha_0,\alpha_1$ the {\it blue edges},
and $\beta_0,\beta_1$ the {\it red edges}. We have $s(\alpha_i)= i=s(\beta_i)$ and $r(\alpha_i)=r(\beta_i)= v$ for $i=0,1$.
This is the same separated graph as in Example \ref{exam:lamplighter}, but we have re-labeled the vertices $w_1$ and $w_2$, and also the edges,
to adapt the notation to the interpretation in terms of symbolic dynamics we are giving here. 

\begin{center}{
			\begin{figure}[htb]
				\begin{tikzpicture}[scale=1.5]
                \node (v) at (0,1)  {$v$};
					\node (0) at (-1,0) {$0$};
					\node (1) at (1,0) {$1$};
				        \draw[<-,blue] (v) -- (0) node[midway, below] {$\alpha_0$};
                        \draw[<-,blue] (v) -- (1) node[midway, below] {$\alpha_1$};
                       \draw[<-,bend right=30,red] (v) to node[midway,above] {$\beta_0$} (0);
                       \draw[<-,bend left=30,red] (v) to node[midway,above] {$\beta_1$} (1);
                     \end{tikzpicture}
				\caption{The separated graph underlying the lamplighter group}
				\label{fig:lampgroup2}
		\end{figure}
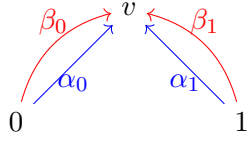}
	\end{center}

Recall from Example \ref{exam:lamplighter} that
$$v\mathcal O (E, C)v\cong C^*(\Z_2\wr \Z)\cong C(\{0,1\}^\Z)\rtimes \Z ,$$ 
and similarly
$$v\Lab (E,C)v\cong \C [\Z_2\wr \Z
	] \cong A(\{0,1\}^\Z)\rtimes \Z, $$ where $\Z_2 \wr \Z$ is the lamplighter group, $C(\{0,1\}^\Z)\rtimes \Z$ is the $C^*$-algebra  associated to the two-sided shift on the Cantor space $\{0,1\}^\Z$ of biinfinite sequences on $\{0,1\}$, and $A(\{0,1\}^\Z)\rtimes \Z$ is the corresponding algebraic crossed product of the Steinberg algebra $A(\{0,1\}^\Z)$ of locally constant functions on $\{0,1\}^\Z$. 

 We here give a dynamical interpretation in terms of the associated canonical sequence $\{(E_n, C^n)\}$ of bipartite separated graphs. 

Set $\mathcal X =\{ 0, 1 \}^{\Z}$ and let $\sigma $ be the usual shift
homeomorphism on $\mathcal X$. Set $\Omega: = \Omega (E,C)$ be the configuration space. 
 
For a finite word $a_1a_2\cdots a_n\in \{ 0,1 \}^n$, we will write
 $$ a_1a_2 \cdots a_{i-1}\underline{a_i} a_{i+1}\cdots a_n := \{ x\in \mathcal X : x_{j-i}= a_j \text{ for } j=1,2,\dots , n \}.$$
These are the cylinder sets, which form a basis of open compact subsets of the topology of $\mathcal X$. 

Observe that, regarding the space $\Omega = \Omega (E,C)$, we may identify $\Omega_v$ with $\mathcal X$, $\Omega_0$ with $\ul{0}$ and $\Omega_1$ with $\ul{1}$ .
We will explain this identification through an example. Observe that a configuration (i.e., a point in $\Omega$) can be described as in 
Figure \ref{fig:lampgroup3}. This is interpreted as a configuration $\xi$ based at $v$ (hence of the form (c.2)), that is, $\xi\in \Omega_v$, where the labeling of $v$ indicates the
choices at $1\in \mathbb F_4$, which in this case are $\alpha_0\in B^0$ and $\beta_1\in R^0$. This configuration is identified with the sequence
$$ x = \cdots 1 \underline{0} 1 1 0 \cdots $$
(Note that we have only specified $5$ values in the biinfinite sequence.)
We have $\xi = \{1,\alpha_0,\beta_1, \alpha_0\beta_0^{-1}, \beta_1\alpha_1^{-1},\dots  \}$. Now $\alpha_0^{-1} \cdot \xi$ is a configuration based at $0$, 
which is also identified with $x$. So, we look at the blue edges as identities, identifying two copies of $\ul{0}$. Now observe that
$\beta_1^{-1} \cdot \xi$ can be identified with $\sigma(x) =  \cdots 1 0\underline{1} 1 0 \cdots $, and thus $\theta_{\beta_1}\colon \ul{1} \to \Omega _v =\mathcal X$ is identified with $\sigma^{-1}$. 
In this way, we think of the red edges at this level as implementing the inverse shift $\sigma^{-1}$
on $\mathcal X$.

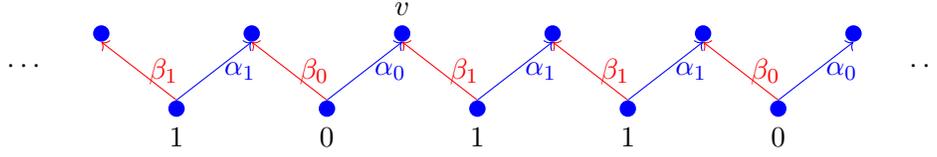
\begin{figure}[htb]
	\begin{tikzpicture}
	      \node[circle,fill=blue,scale=0.6] (A1) at (-5,1) {};  
         \node[circle,fill=blue,scale=0.6] (A2) at (-3,1) {};   
         \node[circle,fill=blue,scale=0.6,label={\text{$v$}}] (A3) at (-1,1) {};
		\node[circle,fill=blue,scale=0.6] (A4) at (1,1) {};
         \node[circle,fill=blue,scale=0.6] (A5) at (3,1) {};
          \node[circle,fill=blue,scale=0.6] (A6) at (5,1) {};
          \node[circle,fill=blue,scale=0.6, label=below:{\text{$1$}}] (B1) at (-4,0) {};
		\node[circle,fill=blue,scale=0.6,label=below:{\text{$0$}}] (B2) at (-2,0) {};
        \node[circle,fill=blue,scale=0.6,label=below:{\text{$1$}}] (B3) at (0,0) {};
        \node[circle,fill=blue,scale=0.6,label=below:{\text{$1$}}] (B4) at (2,0) {};
        \node[circle,fill=blue,scale=0.6,label=below:{\text{$0$}}] (B5) at (4,0) {};
        		\node[label=\text{$\cdots$}] (K) at (-6,0.2) {};
		\node[label=\text{$\cdots$}] (K+1) at (6,0.2) {};
		\draw[<-,red] (A1.south) -- node[below,right]{$\beta_1$} (B1.north);  
        \draw[<-,red] (A2.south) -- node[below,right]{$\beta_0$} (B2.north);  
        \draw[<-,red] (A3.south) -- node[below,right]{$\beta_1$} (B3.north);  
		\draw[<-,red] (A4.south) -- node[below,right]{$\beta_1$} (B4.north);  
        \draw[<-,red] (A5.south) -- node[below,right]{$\beta_0$} (B5.north);  
        \draw[<-,blue] (A2.south) -- node[below,right]{$\alpha_1$} (B1.north);  
        \draw[<-,blue] (A3.south) -- node[below,right]{$\alpha_0$} (B2.north);  
        \draw[<-,blue] (A4.south) -- node[below,right]{$\alpha_1$} (B3.north);  
		\draw[<-,blue] (A5.south) -- node[below,right]{$\alpha_1$} (B4.north);  
        \draw[<-,blue] (A6.south) -- node[below,right]{$\alpha_0$} (B5.north);  
        \end{tikzpicture}
        \caption{A configuration in $(E,C)$}
				\label{fig:lampgroup3}
\end{figure}

 Hence the first layer of the sequence $\{ (E_n, C^n) \}_{n\ge 0 }$ corresponds to a trivial decomposition $\mathcal X^0= \mathcal X$ and to the decomposition
 $$\mathcal X^1_0= \underline{0}, \qquad  \mathcal X^1_1= \ul{1}.$$
 The maps corresponding to the edges are the maps
 $\alpha _i \colon \mathcal X^1_i \to \mathcal X$ and $\beta _i \colon \mathcal X^1_i  \to \mathcal X$ defined by
 $\alpha _i = \text{id}|_{\mathcal X_i^1}$ and $\beta _i= \sigma^{-1}|_{\mathcal X_i^1}$.

 We now describe the clopen sets corresponding to the separated graph $(E_1,C^1)$, which we identify with the vertices of $E_1$. 
 We have $E_1^{0,0} = E_0^{0,1} = \{ \ul{0},\ul{1} \}$, and 
 $$E_1^{0,1} = \{ \ul{0}0, \ul{0}1, \ul{1}0, \ul{1}1 \}. $$
 There is a blue edge, representing inclusion, from $\ul{0}i$ to $\ul{0}$, for each $i=0,1$. 
 Similarly, there is a blue edge from $\ul{1}i$ to $\ul{1}$, for each $i=0,1$.
At this level, the red edges correspond to the shift $\sigma$. For instance, there is a red edge from $\ul{0}1$ to $\ul{1}$. In the notation of Definition \ref{def:F.infty.and.others}, this edge corresponds to $\alpha ^{\beta_1} (\alpha_0)$. 

\begin{center}{
			\begin{figure}[htb]
				\begin{tikzpicture}[scale=1.5]
					\node[circle,fill=blue,scale=0.6, label=above:{\text{$v$}}] (A1) at (0,1) {}; 
                    \node[circle,fill=blue,scale=0.6, label=above:{\text{$0$}}] (B1) at (-2,0) {};
                    \node[circle,fill=blue,scale=0.6, label=above:{\text{$1$}}] (B2) at (2,0) {};
                    \node[circle,fill=blue,scale=0.6, label=above:{\text{$00$}}] (C1) at (-3,-1) {};
                    \node[circle,fill=blue,scale=0.6, label=above:{\text{$01$}}] (C2) at (-1,-1) {};
                    \node[circle,fill=blue,scale=0.6, label=above:{\text{$10$}}] (C3) at (1,-1) {};
                    \node[circle,fill=blue,scale=0.6, label=above:{\text{$11$}}] (C4) at (3,-1) {};
                     \draw[<-,blue] (A1) to (B1);
                     \draw[<-,blue] (A1) to (B2);
                     \draw[<-,blue] (B1) to (C1);
                     \draw[<-,blue] (B1) to (C2);
                     \draw[<-,blue] (B2) to (C3);
                     \draw[<-,blue] (B2) to (C4);
                     \draw[<-,bend right=30,red] (A1) to (B1);
                     \draw[<-,bend left=30,red] (A1) to (B2);
                     \draw[<-,bend right=30,red] (B1) to (C1);
                     \draw[<-,red] (B1) to (C3);
                     \draw[<-,red] (B2) to (C2);
                     \draw[<-,bend left=30,red] (B2) to (C4);
                     \end{tikzpicture}
				\caption{The two first layers of the Bratteli diagram of the full shift}
				\label{fig:lampgroup4}
		\end{figure}
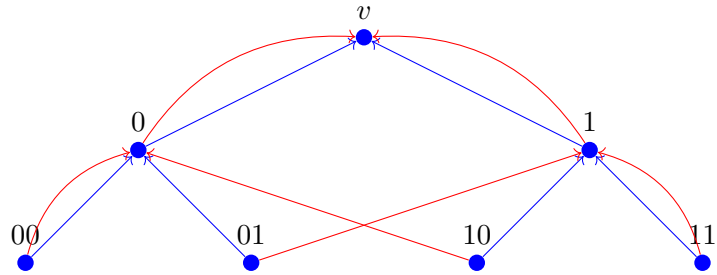}
	\end{center}

 We now describe in general the clopen sets corresponding to the separated graph $(E_n, C^n)$ for any $n\ge 1$. Let $w = a_1a_2\cdots a_n\in \{0,1 \}^n$ be a word of length $n$.
 If $n=2m$ is even, define
 $$\mathcal X^n_w := [a_1\cdots \ul{a_m}{a_{m+1}}\cdots a_{2m}].$$
 If $n=2m+1$ is odd, define
 $$\mathcal X^n_w := [a_1\cdots a_{m}\ul{a_{m+1}}a_{m+2}\cdots a_{2m+1}] .$$
 The set $E_n^{0,0}$ has exactly $2^n$ elements, and thus $E_n^{0,1}=E_{n+1}^{0,0}$ has exactly $2^{n+1}$ elements. The vertices in $E_n^{0,0}$ correspond to a decomposition
 $$\mathcal X= \bigsqcup _{w\in \{ 0,1 \}^n } \mathcal X^n_w. $$

 Let $n= 2m$ and take $w\in \{ 0,1 \}^n$. Then $C^n_w= \{ X^w_1 , X^w_2 \}$, where $X^w_1= \{ \alpha^{n+1}_{0w}, \alpha^{n+1}_{1w} \}$ and $X^w_2 = \{ \beta^{n+1}_{w0} , \beta^{n+1}_{w1} \}$.
 The edges $\alpha^{n+1}_{iw}$ correspond to the maps, denoted in the same way,  $\alpha^{n+1}_{iw} \colon \mathcal X^{n+1}_{iw} \to \mathcal X^n_w$, where each is simply the identity on the respective domain.
 Similarly, the edges $\beta^{n+1}_{wi}$ correspond to the maps  $\beta^{n+1}_{wi} \colon \mathcal X^{n+1}_{wi} \to \mathcal X^n_w$ given by the restriction of $\sigma^{-1}$ to the respective domains.

 Now let $n= 2m+1$ and take $w\in \{ 0,1 \}^n$. In this case, we have  $C^n_w= \{ X^w_1 , X^w_2 \}$,
 where $X^w_1= \{ \alpha^{n+1}_{w0}, \alpha^{n+1}_{w1} \}$ and $X^w_2 = \{ \beta^{n+1}_{0w} , \beta^{n+1}_{1w} \}$.
The edges $\alpha^{n+1}_{wi}$ again correspond to the maps, denoted in the
same way,  $\alpha^{n+1}_{wi} \colon \mathcal X^{n+1}_{wi} \to
\mathcal X^n_w$, acting as the identity on the respective domains, while the edges $\beta^{n+1}_{iw}$ correspond to the maps  $\beta^{n+1}_{iw} \colon \mathcal X^{n+1}_{iw} \to \mathcal X^n_w$ given by the restriction of $\sigma$ to the respective domains.

The separated Bratteli diagram $(F,D)= \cup _{n=0}^{\infty} (E_n,C^n)$ is the superposition of two infinite binary trees, with the same vertex set. One of the binary trees is the {\it blue skeleton}
of $(F,D)$, representing identities, and the other binary tree is the {\it red skeleton} of $(F,D)$, which alternately represents $\sigma^{-1}$ and $\sigma$. 
This construction has been generalized in \cite{ac24} to arbitrary surjective local homeomorphisms on zero-dimensional metrizable compact spaces. See also \cite{claramunt} for a review, through concrete examples, of this construction. 
 \end{example}

See \cite{al18} for more information on this example.

\medskip

We can now compute the type semigroup of the full shift. We will state the result only for a two-element alphabet, but an analogous result holds for any finite alphabet.

For a homeomorphism $\alpha$ of the Cantor set $X$, we will simply denote by $\Typ (\alpha)$ the type monoid $\Typ (X, G ,\mathbb K)$ of the associated action of $\Z$ on $X$. 

If $\alpha $ is a minimal homeomorphism on the Cantor set $X$, then its {\it topological full group} $ [[ \alpha ]]$, which is the group of units of the Boolean inverse semigroup $\pHomeo (X, \Z)$, admits a dense locally finite subgrup (see \cite{GPS95} and \cite{Mell-Robert}). Hence by \cite[Theorem 2.26]{Melleray}, the type semigroup $\Typ (\alpha)$ is cancellative and unperforated, and thus, by e.g. \cite[Theorem 3.14]{ag15}, it is the positive cone of a countable dimension group.

The next theorem shows that the above result cannot be extended to the non-minimal situation, and that the type monoid $\Typ (\alpha)$ might be even a wild refinement monoid.

\begin{theorem}
\label{thm:type-semigroup-lamplighter}
For each $n\ge 1$, set $D_n= \{0,1\}^n$ be the set of words of length $n$ over the alphabet $\{0,1\}$. Let $M_n$ be the commutative monoid generated by $D_n$ with defining relations:
$$a0+a1 = 0a+1a \qquad \text{ for all } a\in D_{n-1}.$$
Let $\varphi_n \colon M_n\to M_{n+1}$ be the monoid homomorphism defined by $\varphi_n (b) = 0b+1b= b0+b1$ for $b\in M_n$. 
Then the type semigroup $\Typ (\sigma)$ of the full two-sided shift on $\{0,1\}$ is naturally isomorphic to $M:= \varinjlim_{n\in \N} (M_n,\varphi_n)$. 
The monoid $M$ and all the monoids $M_n$, for $n\ge 2$, are strongly separative, stably finite non-cancellative monoids. Moreover,  $M$ is a wild refinement monoid, and $M$ is isomorphic to the Murray-von Neumann monoid $ V(K[\Z_2 \wr \Z])$ of the group algebra of the lamplighter, for any field $K$.
\end{theorem}

\begin{proof}
This is a consequence of \cite[Theorem 7.4]{ae14} and its proof. We will show here that there is a well-defined monoid homomorphism $\psi \colon M\to \Typ (\sigma)$, and then we will indicate how to conclude that it is an isomorphism using the above mentioned result. 

Note that for any clopen set $U = a_1a_2 \cdots a_{i-1}\underline{a_i} a_{i+1}\cdots a_n$ in $\mathcal X$, the class of $U$ in $\Typ (\sigma)$ does not depend 
on the particular origin that we choose, because $[\sigma^m(U)] = [U]$ for each $m\in \Z$ in $\Typ (\sigma)$. Hence $[U]$ only depends on the block $a_1a_2\cdots a_n$, and so we have a 
well-defined map $\{0,1\}^n\to \Typ (\sigma)$ from the set of blocks of length $n$ to $\Typ (\sigma)$, sending $a_1a_2\cdots a_n\in D_n$ to $[a_1a_2 \cdots a_{i-1}\underline{a_i} a_{i+1}\cdots a_n]$, which is independent of $i$. This set map induces a monoid homomorphism $\psi_n \colon F(D_n)\to \Typ (\sigma)$, where $F(D_n)$ is the free commutative monoid on $D_n$. 
We need to show that this homomorphism descends to a homomorphism from $M_n= F_n/{\sim_n}$ to $\Typ(\sigma)$, denoted also by $\psi_n$, where $\sim_n$ is the congruence on $F(D_n)$ generated by the pairs
$(a0+a1,0a+1a)$ for all $a\in D_{n-1}$. Let $a=a_1\cdots a_{n-1}\in D_{n-1}$. Then 
\begin{align*}
\psi _n (a0 +a1) & =[\underline{a_1}a_2\cdots a_{n-1}0] + [\underline{a_1}a_2\cdots a_{n-1}1] \\
& = [\ul{a_1}a_2\cdots a_{n-1}] =[0\ul{a_1}a_2\cdots a_{n-1}] + [1\ul{a_1}a_2\cdots a_{n-1}]\\
& =\psi _n (0a+a1).
\end{align*}
Hence the map $\psi_n \colon M_n\to \Typ (\sigma)$ is a well-defined monoid homomorphism. It is also easy to show that the maps $\psi_n$ are compatible with the transition maps $\varphi_n$, that is, we have $\psi_{n} = \psi_{n+1}\circ \varphi_n $ for all $n\ge 1$. Hence we have a well-defined monoid homomorphism $\psi \colon M\to \Typ (\sigma)$ which is determined by the above rules. Note that, using the identification of the canonical resolution of the separated graph $(E,C)$ with the separated Bratteli diagram described in Example \ref{exam:lamplighter2}, it is clear that $M_n=M(E_{n-1},C^{n-1})$ for all $n\ge 1$, and that the maps $\varphi_n\colon M_n\to M_{n+1}$ correspond to the maps $\iota_{n-1} \colon M(E_{n-1},C^{n-1}) \to M(E_{n},C^{n})$ defined in \cite[Lemma 4.5]{ae14}. Under these identifications, the maps $\psi_n$ we have defined coincide with the maps $\Psi_{n-1}$ that appear in the proof of \cite[Theorem 7.4]{ae14}. It follows from that proof that the map $\psi\colon M\to \Typ (\sigma)$ is an isomorphism.       

It follows from Wehrung's Theorem (Theorem \ref{thm:supramenable}) that $M\cong \Typ (\sigma)$ is a strongly separative monoid. Since all the maps $M_n\to M$ are order-embeddings (by \cite[Lemma 4.5]{ae14}), it follows that all the monoids $M_n$ are strongly separative. To show that $M$ and all the monoids $M_n$, for $n\ge 2$, are non-cancellative, it suffices to show that $M_2$ is non-cancellative. But this follows from the identity $00+01 = 00+10$, which holds in $M_2$, together with the observation that $01\ne 10$ in $M_2$.

We now show that each $M_n$ is stably finite, which implies that $M= \varinjlim M_n$ is stably finite. 
To show that $M_n$ is stably finite, it suffices to build a faithful homomorphism $\tau_n \colon M_n \to \Z^+$. The homomorphism
$\tau_n $ is defined first on $F(D_n)$ as the map sending each $b\in D_n$ to $1$. Hence $\tau_n$ sends any element
$\alpha = \sum_{i=1}^n a_i$ to $n$, for $a_i\in D_n$. It is clear that this homomorphism factors through $M_n=F(D_n)/{\sim_n}$, and that it is faithful. Hence $M_n$ is stably finite.

The fact that $M$ is a wild refinement monoid follows from \cite[Theorem 3.14]{ag15}, because $M$ is a stably finite refinement conical monoid which is not cancellative.

The final assertion follows from \cite[Corollary 5.9 and Example 9.7]{ae14}.
\end{proof}

\begin{remark}
\label{rem:work-in-progress}
Although we know that the monoids $M_n$ from Theorem \ref{thm:type-monoid-self-similar} are strongly separative, its fine structure remains elusive. We indeed know that these monoids are strongly separative by an indirect argument, as indicated in the proof of Theorem \ref{thm:type-semigroup-lamplighter}, but I do not know a direct argument just using the presentation of the monoid.
In work in progress, Claramunt, Nazemian and the author tackle the problem of establishing a {\it normal form} for the elements of
$M_n$, which might be a useful tool to understand its fine structure. In particular, it would be interesting to know whether the monoids $M_n$ are unperforated or not.
\end{remark}

\section*{Acknowledgments}{The author is very grateful to Joan Claramunt, Bartosz Kwaśniewski and Fernando Lledó for their comments and suggestions on an earlier version of the paper, that have contributed to improve the presentation.}

\providecommand{\bysame}{\leavevmode\hbox
	to3em{\hrulefill}\thinspace}



\end{document}